\documentclass[12pt]{article}
\usepackage{hyperref} 
\usepackage{amssymb,amsmath,amsthm,mathrsfs}
\usepackage{float}
\usepackage{booktabs}
\usepackage{multirow}
\usepackage{url}
\usepackage{bm}
\usepackage{graphicx}
\usepackage{algorithm,algorithmic,subfigure}
\usepackage{subeqnarray}
\usepackage{cases}
\usepackage{authblk}
\usepackage{blindtext}
\usepackage{xcolor}
\usepackage{cleveref}

\newcommand{\Ical}{{\mathcal{I}}}

\newcommand{\argmin}{\text{argmin}}
\newcommand{\Rank}{\mathrm{Rank}}
\newtheorem{theorem}{Theorem}[section]

\newtheorem{lemma}[theorem]{Lemma}
\newtheorem{assumption}[theorem]{Assumption}
\newtheorem{proposition}[theorem]{Proposition}
 \theoremstyle{definition}
\newtheorem{definition}[theorem]{Definition}
\theoremstyle{remark}
\newtheorem{remark}[theorem]{Remark}
\newtheorem{example}{Example}[section]
\providecommand{\keywords}[1]
{
	\small
	\textbf{\textit{Keywords---}} #1
}

\begin{document}
\date{}

\title{Efficient Low-rank Identification via Accelerated Iteratively Reweighted Nuclear Norm Minimization}

\author[1]{Hao~Wang}
\author[1]{Ye~Wang}
\author[2]{Xiangyu~Yang*}
\affil[1]{School of Information Science and Technology, ShanghaiTech University, Shanghai 201210, China}
\affil[2]{Center for Applied Mathematics of Henan Province, Henan University, Zhengzhou, 450046, China}
\affil[1]{\textit {\{wanghao1,wangye\}@shanghaitech.edu.cn}}
\affil[2]{\textit {yangxy@henu.edu.cn}}
\maketitle

\begin{abstract}
 This paper considers the problem of minimizing the sum of a smooth function and the Schatten-$p$ norm of the matrix. Our contribution involves proposing accelerated iteratively reweighted nuclear norm methods designed for solving the nonconvex low-rank minimization problem. Two major novelties characterize our approach. Firstly, the proposed method possesses a rank identification property, enabling the provable identification of the "correct" rank of the stationary point within a finite number of iterations. Secondly, we introduce an adaptive updating strategy for smoothing parameters. This strategy automatically fixes parameters associated with zero singular values as constants upon detecting the "correct" rank while quickly driving the rest parameters to zero. This adaptive behavior transforms the algorithm into one that effectively solves smooth problems after a few iterations, setting our work apart from existing iteratively reweighted methods for low-rank optimization.  We prove the global convergence of the proposed algorithm, guaranteeing that every limit point of the iterates is a critical point. Furthermore, a local convergence rate analysis is provided under the Kurdyka-{\L}ojasiewicz property. We conduct numerical experiments using both synthetic and real data to showcase our algorithm's efficiency and superiority over existing methods. 
\end{abstract}



\section{Introduction}\label{Sec_Introx}
We consider the following regularized nonconvex matrix optimization problem 
\begin{equation}\label{pro_pri}\tag{$\mathcal{P}$}
	\min\limits_{X\in \mathbb{R}^{m\times n} }  F (X):= f (X)+\lambda\|X\|_{p}^{p},
\end{equation} 
where the loss term $f:\mathbb{R}^{m\times n} \to \mathbb{R}$ is twice continuously differentiable and the regularization term $\Vert X \Vert_{p} = \left( \sum_{i=1}^{\min\{m,n\}} \sigma_i(X) \right)^{1/p}$ is commonly referred to as the nonconvex Schatten-$p$ norm\footnote{It is a matrix quasi-norm when $0 <p <1$. We call it a norm for convenience.} with $p\in(0,1)$, and $\sigma_{i}(X)$ is the $i$th element of the singular value vector of $X$. The parameter $\lambda > 0$ is tunable,  providing a proper trade-off between the loss and regularization terms. Throughout our discussion, we assume without loss of generality that $m \leq n$.

The problem of the form \eqref{pro_pri} incorporates an (approximate) low-rank assumption for the desired solution. Indeed, \eqref{pro_pri} is generally regarded as a potent nonconvex relaxation of the following rank-regularized formulation:
\begin{equation}\label{Eq:RankModel}
	\min\limits_{X \in \mathbb{R}^{m\times n}}  f (X) + \lambda \Rank (X).
\end{equation}
Note that $\Rank(X) = \Vert \bm{\sigma}(X)\Vert_{0}$, and hence the Schatten-$p$ norm of $X$ approximates $\Rank(X)$ in the sense that $\lim_{p \to 0^{+}} \Vert X\Vert_{p} = \Rank(X)$. Problem \eqref{Eq:RankModel} models many important problems that emerged in science and engineering fields, including quality-of-service  (QoS) prediction \cite{LRMM_APP2QoSDATA_IEEE2017}, 
recommender systems \cite{lee2016llorma}
machine learning  \cite{indyk2019learning} 
and image processing  
\cite{LRMR_APP2ImgProcess_TIP2014,LRMR_APP2ImgProcess_SP2020}. However, such a matrix rank minimization problem \eqref{Eq:RankModel}  is a well-known NP-hard problem \cite{review_LRR_2021}.
In this regard, problem \eqref{pro_pri} serves as an efficient alternative to \eqref{Eq:RankModel}. Therefore, it arises in an incredibly wide range of settings throughout science and applied mathematics \cite{lee2016llorma,chiang2018using,jun2019bilinear,tong2021accelerating,pal2023online}. In particular, such an optimization model \eqref{pro_pri} is used in many modern machine learning tasks. Examples include low-rank features learning \cite{wang2019robust}, Multi-View learning \cite{liu2015low}, and Transfer learning \cite{lin2019non}, to name just a few.

A commonly employed approach for addressing \eqref{pro_pri} is the Iteratively Re-Weighted Nuclear  (IRWN) norm-type algorithm, which falls under the majorization-minimization algorithmic framework. The nonsmooth and non-Lipschitz properties of the Schatten-$p$ norm typically prompt researchers to initiate their exploration with a smoothed objective function:
\begin{equation}\label{pro_reg_eps}
	F_{\bm{\epsilon}}(X) := f(X) + \lambda \sum_{i=1}^{m} (\sigma_{i} (X) + \epsilon_i)^{p},
\end{equation}
where $\epsilon_{i} > 0,\forall i\in[m]$ refer  to the perturbation parameters. The modified function $F_{\bm{\epsilon}}(X)$ adjusts $F(X)$ by introducing a perturbation parameter to each singular value.  During the $k$th iteration with the iterate $X^k$,  IRWN effectively generates the new update $X^{k+1}$ by (approximately) solving 
\begin{equation}\label{eq:GeneralSub}
	X^{k+1} \longleftarrow \argmin_{X \in \mathbb{R}^{m\times n}} F_{\textrm{surro}}(X;X^{k}),
\end{equation} 
\begin{equation}\label{eq:GeneralSubObj}
	F_{\textrm{surro}}(X;X^{k}) :=  \langle \nabla f(X^k), X - X^k\rangle + \frac{\mu}{2}\Vert X - X^k \Vert_{F}^2+ \lambda\Vert X \Vert_{* \bm{w}^{k}},
\end{equation}
where the positive number $\mu$ is generally required to exceed the Lipschitz constant of the smooth loss term $f$, and $	\|X\|_{*\bm{w}^k} = \sum_{i=1}^{m} w_{i}^k\sigma_{i}(X)$ serves as a surrogate for $\Vert X\Vert_p^{p}$ at $X^{k}$. Here, $w_i^k \geq 0, \forall i \in [m] $ represents the weight assigned to $\sigma_{i}(X)$: 
$w_i^k=p(\sigma_i(X^k)+\epsilon_i^k)^{p-1}$.

The major difference between variants of IRWN may be the updating rule for the perturbation $\bm{\epsilon}^{k}$, since the values of $\bm{\epsilon}^{k}$ are critically linked to the well-posedness and solvability of the subproblem \eqref{eq:GeneralSub}.  
As is proved in \cite[Theorem 2.2]{chen2013reduced},  to guarantee  $\Vert X\Vert_{*\bm{w}}$ is indeed a convex matrix-norm, the weights must 
be in descending order. This requirement proves impractical within the context of IRWN, since sufficiently small $\bm{\epsilon}^{k}$ leads to weights in ascending order.  On the other hand,  the subproblem is nonconvex, yet it boasts a closed-form optimal solution \cite{IRSVM_LuZhaoSong_2014} when the weights are arranged in ascending order. One simple approach is to maintain $\bm{\epsilon}$ as sufficiently small positive constants during the iteration \cite{opt_simu_svd_2017} to maintain the weights in ascending order. As such, it indeed solves the relaxed problem  \eqref{pro_reg_eps} as its goal. 
It is generally believed that \eqref{pro_reg_eps}  approximates the original problem \eqref{pro_pri} well only for sufficiently small $\bm{\epsilon}$. 
Fixing  $\bm{\epsilon}^{k}$ as sufficiently small values (especially those associated with the zero singular values of the initial points) may cause the algorithm easily trapped into undesired local minimizers near the initial point. 
A natural idea to remedy this strategy is to use the same perturbation value for each singular value, i.e., $\bm{\epsilon}^{k} = \epsilon^k \bm{e}$, and then decrease $\bm{\epsilon}^k$ during the iteration. In this way, the performance of the algorithm critically depends on the speed of driving $\bm{\epsilon}^k$ to zero. It is conceivable that reducing  $\epsilon_i^k$ associated with the zero singular values of the iterates too fast may lead to undesired local minimizers, and reducing  $\epsilon_i^k$ associated with the positive singular values of the iterates too slow may cause the algorithm sluggish.  
An ideal updating strategy should be able to quickly detect those zero singular values in the found optimal solution, and then automatically terminate the decrease for $\epsilon_i^k$ assigned to them and at the same time keep driving other $\epsilon_i^k$ rapidly to zero.  
Another benefit of such a strategy is that the $\epsilon_i^k$ associated with the zero singular value does not affect the objective value near 
the optimal solution, and the algorithm's behavior then only depends on the positive singular value and the decreasing speed of 
the rest $\epsilon_i^k$. However, such an updating strategy may be sophisticated and challenging to design since the user typically lacks prior knowledge of the rank of the final solution until the entire problem is resolved. 

In this paper, we propose an Extrapolated Iteratively Reweighted  Nuclear norm with Rank Identification (EIRNRI) to solve \eqref{pro_pri}.  We first add perturbation parameters $\epsilon_i$ to each singular value of the matrix to smooth the Schatten-$p$ norm.  Then we construct the weighted nuclear ``norm''\footnote{It is indeed not a norm since it is not convex.} subproblem of the approximated function combined with an extrapolation technique.  An adaptive updating strategy for $\bm{\epsilon}$ is also designed, such that it can automatically terminate the update for $\epsilon_i$ associated with the zero singular values and drive those associated with the positive singular values quickly to zero.
This updating strategy keeps the weights in ascending order so that the subproblem is nonconvex but has a closed-form optimal solution \cite[Theorem 3.1]{IRSVM_LuZhaoSong_2014}.
Our algorithm is designed in a way such that after finite iterations, the algorithm can automatically detect those zero singular values in the optimal solution, meaning the algorithm eventually behaves like solving a smooth problem in the Grassmann manifold. 
Based on this,  the local convergence rate can be easily derived. It should be mentioned that our work mainly considers applying the proposed algorithm to solve the representative Schatten-$p$ regularized problem \eqref{pro_pri}, however, it is important to note that the proposed algorithm can be extended to other nonconvex regularization functions of the singular values quite straightforwardly, including the Logarithm \cite{friedman2012fast}, Exponential-type Penalty \cite{gao2011feasible}, Geman \cite{geman1995nonlinear}, Laplace \cite{trzasko2008highly}, Minimax Concave Penalty \cite{zhang2010nearly} and Smoothly Clipped Absolute Deviation (SCAD) \cite{fan2001variable}, and more. 

\subsection{Related Work}
Over the last decade, significant attention has been directed towards low-rank optimization, yielding theoretically and practically efficient algorithms applicable to diverse problems in signal processing and modern machine learning. Within the extensive body of work, we specifically review the most relevant works.

\textbf{IRWN-type algorithms.} The work of \cite{opt_simu_svd_2017} proposed a proximal iteratively reweighted nuclear norm (PIRNN) algorithm. Their algorithm adds a prescribed positive perturbation parameter $\epsilon_{i}$ to each singular value $\sigma_{i} (X)$ and \emph{fixed} it during the iteration of the algorithm.
Therefore, it indeed solves the relaxed problem  \eqref{pro_reg_eps} as its goal. As a stark contrast,  our method is designed for the original problem \eqref{pro_pri} in the sense that the perturbation parameter is automatically driven to $0$
so that the iterates can successfully recover the rank of the first-order optimal solution to \eqref{pro_pri}. 
It should be stressed that our updating strategy is designed such that $\bm{\epsilon}$ is decreased to zero in an appropriate speed to maintain the the well-posedness of the subproblems and the convergence rate of the overall algorithm. 

The immediate predecessor of our work, to the best of our knowledge, is the Iteratively Reweighted Nuclear Norm (IRNN) algorithm proposed in \cite{ge_nn_LRMM_Canyi_2014} and its acceleration (AIRNN) introduced in \cite{Alg_conflict_AIRNN_2021,ge2022fast}.
IRNN considered a general concave singular value function $g(\sigma_{i}(X))$ as the regularization term. 
It first calculates the so-called supergradient of Schatten-$p$ norm $w_{i}^{k} \in \partial {g(\bm{\sigma}(X^{k}))}$ and uses it as the weight to form the subproblem.
In contrast to our method, this method does not involve the perturbation parameter $\bm{\epsilon}$; therefore, the weight may tend to extreme values as the associated singular value is close to $0$. (As for the zero singular value, this method uses an extremely large constant as the weight).
We suspect this might be the reason for the observation ``IRNN may decrease slowly since the upper bound surrogate may be quite loose.'' reported in \cite{ge_GSVT_LRMM_Canyi_2015}.
Then AIRNN used the extrapolation technique and computed the SVD of a smaller matrix at each iteration to accelerate IRNN.
The biggest difference between our algorithm to IRNN, AIRNN, and other contemporary reweighted nuclear norm methods is the rank identification property possessed by our algorithm, meaning the algorithm can identify the rank of the converged solution after finite iterations.  We elaborate on this in the next subsection. 

\textbf{Rank identification.} The major novelty of our work is the rank identification property of the proposed method, which is an extension of 
the model identification for vector optimization. 
In sparse optimization such as the LASSO or the support-vector machine,  problems generally generate solutions onto a low-complexity model such as solutions of the same supports.
For LASSO, a solution $x^{*}$ typically has only a few nonzeros coefficients: it lies on the reduced space composed of the nonzeros components (the support) of $x^*$. 
Model identification relates to answering the question of whether an algorithm can identify the low-complexity active manifold in finite iterations.
It has become a useful tool in analyzing the behavior of algorithms and has attracted much attention in the past decades in the research of machine learning algorithms in vector optimization. 
For example, coordinate descent \cite{klopfenstein2020model, ModelID_LASSO_AccCD_2018} for convex sparse regularization problems are proved to have model identification, and the convergence analysis is easily derived under this property.
In the last few years, proximal gradient algorithm \cite{ModelID_ProxGD_L1_2011, ModelID_LLcvgc_L1_2014,ModelID_LLcvgc_L1_2017} have been shown the model identification for the $\ell_{1}$ regularized problem.
Recently, the iteratively reweighted $\ell_{1}$ minimization for the $\ell_{p}$ regularized problem is also shown \cite{Zeng_Acc_2022,zenghao_lp2l1_adaEps_paper1}   to have model identification property.  This property also belongs to the research line of active-manifold identification in nonsmooth optimization \cite{lewis2002active, hare2007identifying}.  

While IRWN-type algorithms have been extensively studied, their capability for rank identification has received limited attention. A recent contribution by \cite{lee2023accelerating} explored the use of the proximal gradient method to identify the correct rank for a nuclear norm regularized problem. Notably, their algorithm can serve as a suitable subproblem solver for our approach. In a related vein, the work of \cite{zeng2023proximal} extended the lower bound theory of nonconvex $\ell_{p}$ minimization to Schatten-$p$ norm minimization and incorporates it as a prior in algorithm design. However, the rank identification property of their algorithm remains unverified.

In this paper,  we formalize the rank identification property as follows.

\begin{definition}[Rank identification property] An algorithm is said to possess the rank identification property if and only if for a sequence (or at least a subsequence) $\{X^k\}_{k \in \mathbb{N}_{+}}$ generated by the algorithm converges to a solution $X^*$, then there exists $K \in \mathbb{N}_{+}$ such that for each $k \geq K$, $X^k \in \mathcal{M}(X^*):= \{X \in \mathbb{R}^{m\times n}\mid \Rank(X) = \Rank (X^{*})\}$.
\end{definition}

Our algorithm is designed to possess this property, meaning the singular values of the generated iterates satisfy $\sigma_{i}(X^{k}) = 0,i\in\mathcal{Z}(X^{*})$ and $\sigma_{i}(X^{k}) > 0,i\in\mathcal{I}(X^{*})$ for all sufficiently large $k$, where $\mathcal{Z}(X^{*})$ is the set of indices corresponding the the zero singular values in the optimal solution and  $\mathcal{I}(X^{*})$  corresponds the nonzero singular values.  
Based on this, a adaptively updating strategy of $\epsilon$ can be straightforwardly designed to drive $\epsilon_i, i\in\mathcal{I}(X^{*})$ quickly to zero and automatically cease the updating for $\epsilon_i, i\in\mathcal{Z}(X^{*})$.  
In essence, this implies the algorithm behaves like solving a smooth problem in a low-complexity manifold, facilitating a straightforward derivation of global convergence analysis and application of acceleration techniques. 
To our knowledge,  this idea of designing an algorithm with model/rank identification property for the Schatten-$p$ norm is novel in the context of matrix optimization problems.  

\subsection{Contribution}
We summarize our main contributions in the following.  
\begin{itemize}
	\item We propose an iteratively reweighted nuclear norm minimization method for the nonconvex regularized problem, and the extrapolation techniques are also incorporated into the algorithm to further enhance its performance. 
	
	\item The key novelty of the proposed method is the adaptively updating strategy for updating the perturbation parameters, bringing two benefits: (i) automatic identification of parameters associated with zero and nonzero singular values, enabling the use of tailored update strategies for each. (ii) consistent maintenance of weights in ascending order, ensuring the explicit computation of a global minimizer for the (nonconvex) subproblem.
	
	\item We show that the algorithm possesses a rank identification property, which can successfully identify the 
	rank of the optimal solutions found by the algorithm within finite iterations.  This property, which is barely studied by the existing related work, signifies a distinct contribution. It implies a transition of the optimization problem to a smoother form in the vicinity of the optimal solution.

	\item Global convergence and local convergence rate under the Kurdyka-{\L}ojasiewicz (KL) property are derived for the proposed algorithm. 
	
\end{itemize} 

\subsection{Notation and Preliminaries}
Throughout the paper, we restrict our discussion to the Euclidean space of $n$-dimensional real vectors, denoted $\mathbb{R}^{n}$, and the Euclidean space of $m\times n$ real matrices, denoted $\mathbb{R}^{m\times n}$, where $m,n\in\mathbb{N}$.   $\mathbb{R}_{+}^{n}$ represents the non-negative orthant in $\mathbb{R}^{n}$ and $\mathbb{R}_{++}^{n}$ denotes the interior of $\mathbb{R}^{n}_{+}$. $\mathbb{R}_{\uparrow}^{n}$ and  $\mathbb{R}_{\downarrow}^{n}$ are used to indicate the set of non-decreasingly ordered vectors and non-increasingly ordered vectors, respectively. Additionally, we use the notation $[n] = \left\{1,2,\cdots, n \right\}$ to denote the integer set from $1$ to $n$, for any $n \in \mathbb{N}$.  
For any $\bm{x}, \bm{y} \in\mathbb{R}^{n}$,
the element-wise (Hadamard) product between $\bm{x}$ and $\bm{y}$ is given by $(\bm{x}\circ \bm{y})_{i} = x_{i}y_{i}$ for $i \in [n]$. By abuse of notation, $\circ$ is also used to denote function composition. Define the $\ell_{p}$-(quasi)-norm of $x \in \mathbb{R}^{n}$ as $\Vert x \Vert_{p} = (\sum_{i=1}^{n}\vert x_i\vert^{p})^{1/p}$.

For any $X, Y \in \mathbb{R}^{m \times n}$ (assuming $m\le n$ for convenience), the Frobenius norm of $X$ is denoted as $\left\|X\right\|_{F} $, namely, $\Vert X\Vert_{F} = \left(\sum_{i=1}^{m}\sum_{j=1}^{n}\vert X_{ij}\vert^2 \right)^{1/2}  = \mathrm{tr} \left(X^{\top}X\right)^{1/2}$.
The Frobenius inner product is $\langle X,Y\rangle=\mathrm{tr} \left(X^{\top}Y\right) $. 
Let $\mathrm{diag}(\bm{x}) $ denote the diagonal matrix with vector $\bm{x}$ on its main diagonal and zeros elsewhere.
The full singular value decomposition (SVD) \cite{general_svd_dec} of $X\in\mathbb{R}^{m \times n}$ is
\begin{equation*}\label{eq_svd_sorted}
	X = U\mathrm{diag} \left(\bm{\sigma} (X)\right) V^{\top},
\end{equation*}
where $(U,V) \in \overline{\mathcal{M}}(X)$ with $\overline{\mathcal{M}}(X):=\{(U,V) \in \mathbb{R}^{m\times n} \times \mathbb{R}^{m\times n} \mid U^{\top}U = V^{\top}V =I, X = U\textrm{diag}(\sigma(X))V^{\top}\}$  and $\bm{\sigma}(X) \in \mathbb{R}_{\downarrow}^{m} \cap \mathbb{R}_{+}^{m}$ denotes the singular value vector of $X$. Suppose ${\Rank}(X) = r \leq m$. The associated thin SVD of $X $ is $X = U_{r}  {\rm diag}(\bm{\sigma}_{r}(X)) V_{r}^{\top}$, where $U_{r}$ and $V_{r} $ are the first $r$ columns of $U$ and $V$, respectively, and $\bm{\sigma}_{r}(X) \in \mathbb{R}_{\downarrow}^{r} \cap \mathbb{R}_{+}^{r}$.

For analysis, we summarize the simultaneous ordered SVD of two matrices introduced in \cite{NA_sgv_1}.

\begin{definition}[Simultaneous ordered SVD]\label{Def:simultaneousSVD}
	We say that two real matrices $X $ and $Y $ of size $m\times n$ have a simultaneous ordered singular value decomposition if there exist orthogonal matrices $U\in\mathbb{R}^{m\times m} $ and $V\in\mathbb{R}^{n\times n} $ such that 
	\begin{equation*}
		X=U\mathrm{diag} (\bm{\sigma} (X))V^{\top} \quad\text{and}\quad Y=U\mathrm{diag} (\bm{\sigma} (Y))V^{\top}.
	\end{equation*}
\end{definition}
\noindent In addition, we define two index sets as follows to track the singular values of the iterates conveniently, which reads 
\begin{equation*}
	\mathcal{I}(X) := \left\{i:\sigma_i(X)>0\right\} \quad\text{and}\quad 
	\mathcal{Z}(X) := \left\{i:\sigma_i(X)=0\right\}. 
\end{equation*}

For a lower semi-continuous function $J:\mathbb{R}^{N}\to (-\infty,+\infty] $, its domain denoted by ${\rm dom} (J):=\{x\in\mathbb{R}^{N}:J (x)<+\infty\}$.  We first recall some concepts of subdifferentials that are commonly used in variational analysis and subdifferential calculus, which is a useful tool in developing optimality conditions of the concerned optimization problem in nonsmooth analysis. 

\begin{definition}[Subdifferentials] Consider a proper lower semi-continuous function. $\varphi :\mathbb{R}^{N}\to  (-\infty,+\infty] $.
	\begin{itemize}
		\item[(i)] The Fr\'{e}chet subdifferential $\widehat{\partial}\varphi $ of $\varphi$ at an $x\in\mathrm{dom}\ \varphi$ is analytically defined as
		\begin{equation*}
			\widehat{\partial}\varphi (x) := \left\lbrace v \in \mathbb{R}^{N}\mid\lim\limits_{u\to x}\inf\limits_{u\neq x} \frac{\varphi (u) -\varphi (x)-\langle v,u-x \rangle }{\|u-x\|_{2}} \ge 0\right\rbrace .      
		\end{equation*}
		
		\item[(ii)] The (limiting) subdifferential of $\partial \varphi$ of $\varphi$ at an $x\in \textrm{dom}\ \varphi$ is defined through the following closure process, which reads,
		\begin{equation*}
			\partial \varphi (x) := \left\lbrace v\in\mathbb{R}^{N}\mid\exists v^{k} \to v, x^{k} \overset{\varphi}{\to} x  \textrm{ with } v^{k}\in\widehat{\partial}\varphi (x^{k}) \textrm{ for all } k \right\rbrace .
		\end{equation*}
		where $x^{k} \overset{\varphi}{\to} x$ refers to $\varphi$-attentive convergence in analysis, meaning $x^{k} \to x$ with $\varphi(x^{k}) \to \varphi(x)$. 
	\end{itemize}
	We mention that $\widehat{\partial} \varphi (x)= \partial \varphi = \emptyset $ for $x\notin \mathrm{dom}\ \varphi$. 
	
\end{definition}

We next collect the result on the limiting subdifferential of the singular value function established in \cite{NA_sgv_1}. Upon that, we present the limiting subdifferential associated with the nonconvex Schatten-$p$ norm.

\begin{lemma}[Limiting subdifferential of singular value function]\label{Lemma:limitingsub_svd} \label{lem_partial_sp,opt_simu_svd_2017} Let $\varphi:\mathbb{R}^{n}\to\mathbb{R}$ be an absolutely symmetric function, meaning $\varphi (x_{1},\cdots,x_{n}) = \varphi(|x_{\pi (1)}|,\cdots,|x_{\pi(n)}|) $ holds for any permutation $\pi$ of $[n]$, and let $\bm{\sigma}(X)$ be the singular values of a matrix $X\in\mathbb{R}^{m\times n}$($n \leq m$ is assumed for convenience). Then the limiting subdifferential of singular value function $\varphi\circ\sigma $ at a matrix $X$ is given by
	\begin{equation*}
		\partial[\varphi\circ\bm{\sigma}] (X) = U\mathrm{diag} \left(\partial \varphi[\bm{\sigma} (X)]\right)V^{\top}, 
	\end{equation*}
	with $U\mathrm{diag} (\bm{\sigma} (X))V^{\top} $ being the SVD of $X$.
\end{lemma}

\noindent A direct consequence of Lemma \ref{Lemma:limitingsub_svd} is the result of limiting subdifferential of $\|X\|_{p}^{p} $.
\begin{proposition}\label{Prop_Psubdifferential}
	Let $\Rank(X) = r \leq m \in \mathbb{N}$. The limiting subdifferential of $\|\cdot\|_{p}^{p}:\mathbb{R}^{m \times n}\to \mathbb{R}$ at a matrix $X$ is given by
	\begin{equation}\label{eq:subdifferentialoflp}
		\begin{aligned}
			\partial \|X\|_{p}^{p} = \partial\left(\sum_{i=1}^{r} \sigma_{i}(X)^{p}\right) &= \partial\left([\Vert\cdot\Vert_{p}^{p} \circ \bm{\sigma}](X) \right)\\
			&=  \left\{U\textrm{diag}(	\Sigma) V^{\top}\mid \Sigma \in (\partial \Vert\bm{\sigma} (X)\Vert_{p}^{p} \circ \partial\vert\sigma_{i}(X)\vert) \right\},
		\end{aligned}      
	\end{equation}
	where $\partial\Vert\bm{\sigma} (X)\Vert_{p}^{p}= \{\vartheta \in \mathbb{R}^{m}\mid \vartheta_j = p\sigma_{j}(X)^{p-1}, j\in[r]\}$ and $(U,V) \in \overline{\mathcal{M}}(X)$.
\end{proposition}

We use the following subdifferential-based stationary principle to establish the first-order necessary optimality conditions of \eqref{pro_pri}.
\begin{theorem}[Nonsmooth versions of Fermat's rule]\label{theo_Fermat}
	If problem \eqref{pro_pri} has a local minimum at $\bar{X}$, then $\bm{0} \in \widehat{\partial}F(\bar{X}) \subset \partial F(\bar{X})$. 
\end{theorem}
Using \Cref{Prop_Psubdifferential} and by \Cref{theo_Fermat}, we define the critical point of \eqref{pro_pri} as follows.
\begin{definition}[Critical point]\label{Def_Stationary} 
	We say that an $X \in \mathbb{R}^{m \times n}$ is a critical point of \eqref{pro_pri} if it satisfies $\bm{0} \in\partial{F}(X)$. Moreover, the set of all critical points of \eqref{pro_pri} is denoted by
	\begin{equation}\label{opt_pri_point}
		\mathrm{crit} (F) := \left\{X\in\mathbb{R}^{m\times n}\mid\mathbf{0}\in \nabla f (X) + \lambda U\mathrm{diag} (\partial|\bm{\sigma} (X)|_{p}^{p})V^{\top},\ (U,V)\in\overline{\mathcal{M}}(X) \right\}. 
	\end{equation}
\end{definition}

The Kurdyka-{\L}ojasiewicz (KL) property plays an important role in our convergence analysis. We next recall the essential components as follows. First, let $\Omega \subset \mathbb{R}^{m\times n}$ and $X \in \mathbb{R}^{m\times n}$, the distance from $X$ to $\Omega$ is defined by
\begin{equation*}
	\mathrm{dist}(X,\Omega):= \inf\{\Vert X - Y\Vert_{F}\mid Y\in\Omega\}.
\end{equation*}
In particular, we have $\mathrm{dist}(X,\Omega) = +\infty$ for any $X$ when $\Omega = \emptyset$.
Next, we provide the definition of the desingularizing function.
\begin{definition}[Desingularizing function]\cite{KL_definition_Hilbert}\label{Def_desingularizing}
	Let $\eta > 0$. We say that $\Phi:[0,\eta] \to \mathbb{R}_{+}$ is a desingularizing function if
	\begin{itemize}
		\item[(i)] $\Phi(0)=0$;
		\item[(ii)] $\Phi$ is continuous on $[0,\eta]$ and of class $C^{1}$ on $(0,\eta)$;
		\item[(iii)] $\Phi'(s) > 0$ for all $s \in (0,\eta)$.
	\end{itemize}
\end{definition}
Typical examples of desingularizing functions are the functions of the form $\Phi(t) = cs^{1-\theta}$, for $c>0$ and $\theta\in [0,1)$. 

Now we define the Kurdyka-{\L}ojasiewicz property.
\begin{definition}[KL property]\label{def_KL_fun}
	Let $F:\mathbb{R}^{m\times n} \to \mathbb{R}\cup\{+\infty\}$ be proper lower-semicontinuous. We say that $F$ satisfies the Kurdyka-{\L}ojasiewicz property at $\bar{X}\in\mathrm{dom}(\partial F) := \{ X \in \mathbb{R}^{m\times n} \mid \partial F(X) \neq \emptyset \} $ if there exists $\eta>0$, a neighborhood $\mathbb{U} (\bar{X},\rho) $ of $\bar{X}$, and a concave desingularizing function $\Phi:[0,\eta) \to \mathbb{R}_{+}$, such that the Kurdyka-{\L}ojasiewicz inequality
	\begin{equation}\label{ineq_KL}
		\Phi'\left(F (X)-F (\bar{X})\right) \mathrm{dist} \left(0,\partial F(\bar{X})\right)\ge 1
	\end{equation}
	holds, for all $X$ in the strict local upper level set
	\begin{equation*}
		\mathrm{Lev}_{\eta}(\bar{X},\rho) := \{X \in \mathbb{U} (\bar{X},\rho) \mid F(\bar{X}) < F(X) < F(\bar{X}) + \eta\}.
	\end{equation*}
\end{definition}
If $F$ satisfies the KL property at any $X \in \mathrm{dom}(\partial F)$, we then call $F$ a KL function.

Moreover, due to the existence of multiple critical points in nonconvex optimization, the KL property at a single point $\bar{X}$ may be insufficient at times. It is necessary for us to introduce the definition of uniform KL property.

\begin{definition}[Uniform KL property]\cite{bolte2014proximal}\label{Uniform_KL}
	Let $\Omega$ be a compact set and let $F:\mathbb{R}^{m\times n} \to \mathbb{R}\cup\{+\infty\}$ be proper lower-semicontinuous function. Assume that $F$ is constant on $\Omega$ and satisfies the KL property at each point of $\Omega$. We say that $F$ has uniform KL property on $\Omega$ if there exist $\epsilon > 0$, $\eta > 0$ and $\Phi$ defined in \Cref{def_KL_fun} such that the KL inequality \eqref{ineq_KL} holds for any $\bar{X} \in \Omega$ and any $X \in \{X \in \mathbb{R}^{m\times n}\mid \mathrm{dist}(X,\Omega) < \epsilon\} \cap \{X \in \mathbb{R}^{m\times n}\mid F(\bar{X}) < F(X) < F(\bar{X}) + \eta\}$.
\end{definition}
\section{Proposed Extrapolated Iteratively Reweighed Nuclear Norm Algorithm with Rank Identification}
In this section, we provide the details of the proposed EIRNN framework for solving \eqref{pro_pri} by discussing the subproblem solution and a novel updating strategy of the perturbation parameter to enable the rank identification property.

Before presenting the proposed algorithm, we make the following assumptions on \eqref{pro_pri} as follows throughout.
\begin{assumption}\label{ASM_fLip} 
	The function $f:\mathbb{R}^{m\times n}\to\mathbb{R} $ is $L_f$-smooth, i.e.,
	\begin{equation*}
		\|\nabla f (X) - \nabla f (Y)\|_{F} \le L_{f}\|X-Y\|_{F}, \ \forall X,Y\in\mathbb{R}^{m\times n},
	\end{equation*}
	where the modulus $L_f \geq 0$ refers to the smoothness parameter.
\end{assumption}
\begin{assumption}\label{ASM_levelBoundness}
	Throughout we assume $F$ is level-bounded \cite[Definition 1.8]{rockafellar2009variational}. This assumption on $F$ corresponds to $\lim_{X\in\mathbb{R}^{m\times n}:\Vert X \Vert \to \infty} F(X) = +\infty$ and further implies $\min_{X\in\mathbb{R}^{m\times n}} F(X) = \underline{F} > -\infty$ and $\{X  \mid \argmin_{X\in \mathbb{R}^{m\times n} } F(X)\} \neq \emptyset$ regarding \eqref{pro_pri}.
\end{assumption}

We state the proposed algorithm in \Cref{alg_acc}, which consists of solving a sequence of weighted nuclear norm regularized subproblems, an extrapolation technique, and an adaptive perturbation parameter updating strategy. 

\begin{algorithm}[htbp]
	\caption{Extrapolated Iteratively Reweighted Nuclear Norm with Rank Identification (EIRNRI)}
	\label{alg_acc} 
	\begin{algorithmic}[1]
		\REQUIRE{ $X^{0}\in\mathbb{R}^{m\times{n}}, \bm{\epsilon}^{0}\in\mathbb{R}^{m}_{++} \cap \mathbb{R}^m_{\downarrow}$, $\mu\in(0,1)$, and $\alpha_{0} \in [0,\bar{\alpha}]$ by \eqref{alpha_update}}.
		\ENSURE{$k = 0$ and $X^{-1}=X^{0}$. } 
		\REPEAT{
			\STATE{Compute weights $w_{i}^{k} = p \left(\sigma_i(X^{k})+\bm{\epsilon}^{k}_{i}\right)^{p-1},\ \forall i \in [m].$
			}
			\STATE{Compute $Y^{k+1}$ according to the extrapolation \eqref{eq: extrapolatedPoint}.}
			\STATE{Compute the new iterate as the solution of  \eqref{eq:NewIter}.}
			\STATE{Update $\bm{\epsilon}^{k}$ by the subroutine of  \Cref{Algo:update_eps}.} 
			\STATE{Choose $\alpha_{k} \in [0,\bar{\alpha}]$.} 
			\STATE{Set $k\gets k+1 $. }
		}
		\UNTIL{ convergence }
	\end{algorithmic}
\end{algorithm}

\subsection{A Weighted Nuclear Norm Surrogate with Extrapolation}

Our presented approach for solving \eqref{pro_pri} is primarily motivated by the substantial literature on proximal gradient-type methods employing acceleration techniques \cite{intro_TK2019} and iteratively reweighted techniques \cite{NonNon_AdaIRWM_zhang_wang_2021}.  Specifically, we first add perturbation  parameters $\bm{\epsilon} \in \mathbb{R}^{n}_{++}$ to each singular value of the matrix to smooth the $p$-th power of the Schatten-$p$ norm, 
\begin{equation}\label{pro_reg_eps_our}
	F(X;\bm{\epsilon}) := f(X) + \lambda \sum_{i=1}^{m} (\sigma_{i} (X) + \epsilon_i)^{p}. 
\end{equation}
Obviously, $F(X;0) = F(X)$.  Drawing upon the Nesterov's acceleration technique \cite{nesterov1983method,Alg_conflict_AIRNN_2021}, our approach  begins by computing an extrapolated $Y^{k}$, using the current iterate $X^{k}$ and the previous one $X^{k-1}$, i.e.,
\begin{equation}\label{eq: extrapolatedPoint}
	Y^{k} = X^{k} + \alpha_{k}(X^{k} - X^{k-1}),
\end{equation}
where $\alpha_{k} \in [0,\bar{\alpha})$ refers to the extrapolation   parameter  and is selected according to the following rule   \cite{Zeng_Acc_2022} 
\begin{equation} \label{alpha_update} 
	\begin{cases}
		\bar{\alpha}\in (0,1), & \text{ if } f (x) \text{ is convex and $L_f$-smooth, } \\
		\bar{\alpha}\in (0,\sqrt{\frac{\beta}{\beta+3L_{f}}}), & \text{ if } f (x) \text{ is $L_f$-smooth.}
	\end{cases}
\end{equation}

At $Y^{k}$, it follows for any feasible $X$ that the perturbed objective $F(X;\epsilon)$ admits a useful upper bound presented below, i.e.,  
\begin{equation}
	\begin{aligned}
		F(X,\bm{\epsilon}) & := f(X) + \lambda \sum_{i=1}^{m}(\sigma_{i}(X) + \epsilon_{i})^{p}\\
		&  \overset{(a)}{\leq} f(Y^{k}) + \langle \nabla f (Y^{k}), X-Y^{k} \rangle + \frac{L_f}{2}\Vert X-Y^{k}\Vert_{F}^{2}\\
		&\quad  + \lambda\sum\limits_{i=1}^{m}p(\sigma_{i}(X^{k}) + \epsilon_{i}^{k})^{p-1}(\sigma_{i}(X) - \sigma_{i}(X^{k}))\\
		&\overset{(b)}{\leq} f(Y^{k}) + \langle \nabla f (Y^{k}), X - Y^k \rangle + \frac{L_f}{2}\Vert X-Y^{k}\Vert_{F}^{2} + \frac{L_f}{2}\Vert X-X^{k}\Vert_{F}^{2}\\
		& \quad + \lambda \sum\limits_{i=1}^{m}p(\sigma_{i}(X^{k}) + \epsilon_{i}^{k})^{p-1}(\sigma_{i}(X) - \sigma_{i}(X^{k})),\\
	\end{aligned}
\end{equation}
where $(a)$ is a direct consequence of $L_f$-smoothness of $f$ under \Cref{ASM_fLip} and the concavity of $(\cdot)^{p}$, and $(b)$ naturally holds due to the nonnegativity of the proximal term $\frac{L_f}{2}\Vert X-X^{k}\Vert_{F}^{2}$. Omitting the constants on the right-hand side, we obtain the following surrogate function $L(X;X^{k},Y^{k},\bm{\epsilon}^{k})$ to approximate $F(X)$ at $X^k$, i.e.,
\begin{equation}\label{pro_Linear}
	L(X;X^{k},Y^{k},\bm{\epsilon}^{k}):= f (Y^{k}) + \langle X, \nabla f(Y^{k}) \rangle+ \frac{\beta}{2}\| X-Y^{k}\|_{F}^{2} + \frac{\beta}{2} \| X-X^{k}\|_{F}^{2} + \lambda\sum\limits_{i=1}^{m}w_{i}^{k}\sigma_{i} (X),
\end{equation}
where $w_{i}^{k} = w(\sigma_{i}(X^{k}),\epsilon_{i}^{k}) = p \left(\sigma_i(X^{k})+\epsilon_{i}^{k} \right)^{p-1}, \forall i \in [m]$ and $\beta > L_f > 0$. The new iterate $X^{k+1}$ is computed as  $L(X^{k};X^{k},Y^{k},\bm{\epsilon}^{k})$, i.e., 

\begin{equation}\label{eq:NewIter}
	\begin{aligned}
		X^{k+1} &\in \argmin_{X \in \mathbb{R}^{m\times n}} L(X;X^{k},Y^{k},\bm{\epsilon}^{k})\\
		&\overset{ }{=} \argmin_{X \in \mathbb{R}^{m\times n}} \left\lbrace \frac{\beta}{2}\left\Vert X - \left(\frac{X^{k}+Y^{k}}{2}- \frac{\nabla f(Y^{k})}{2\beta}\right)\right\Vert_{F}^{2} + \frac{\lambda}{2}\sum\limits_{i=1}^{m}w_{i}^{k}\sigma_{i} (X)\right\rbrace.
	\end{aligned}
\end{equation}

\subsection{Subproblem Solution}
It should be mentioned that solving such a weighted nuclear norm minimization problem \eqref{eq:NewIter} is not a direct extension of solving a weighted $\ell_{1}$ norm minimization counterpart in the vector case. In fact, problem \eqref{eq:NewIter} is nonconvex and hence generally posts challenges to find the global minimizer. As shown in \cite[Theorem 2.2]{chen2013reduced}, the weighted nuclear norm of $X\in \mathbb{R}^{m\times n}$, $\|X\|_{*\bm{w}}$, is convex with respect to $X$ if and only if both the singular value $\sigma_{i}(X)$ and its corresponding weights $w_i$, $\forall i \in [m]$ are in descending order.  However, it is unrealistic to design such a strategy, since $w(\sigma_i(X), \epsilon_i) = p(\sigma_i(X)+ \epsilon_i)^{p-1} <  p(\sigma_j(X)+ \epsilon_j)^{p-1} = w(\sigma_j(X), \epsilon_j)$ for $\sigma_i(X) > \sigma_j(X)$ as $\epsilon_i\to 0$ and $\epsilon_j \to 0$.
Nonetheless, a closed-form global optimal solution to \eqref{eq:NewIter} is available by imposing the ascending order on all weights $w_i, \forall i \in [m]$ \cite{IRSVM_LuZhaoSong_2014,opt_simu_svd_2017}. We restate such a result for  \eqref{eq:NewIter}  in the following theorem.

\begin{theorem}\label{Theo:SubproblemSolve}
	Consider $\eqref{eq:NewIter}$. Let $\bm{w}^{k} \in \mathbb{R}_{\uparrow}^{m} \cap \mathbb{R}^{m}_{++}$, that is,
	\begin{equation}\label{eq:AscendingWeights}
		0<w_1^{k}\leq w_2^k \leq \ldots \leq w_m^{k}.
	\end{equation}
	Then a global optimal solution to  \eqref{eq:NewIter} reads
	\begin{equation}\label{eq:OptSub}
		X^{k+1} = U^{k+1} \text{diag}\left(\left[\Sigma^{k+1}_{i}-\frac{\lambda w_{i}^{k}}{2\beta}\right]_{+}\right) {V^{k+1}}^{\top}
	\end{equation}
	with $U^{k+1}\textrm{diag}(\Sigma^{k+1}){V^{k+1}}^{\top}$ being the SVD of the matrix $\frac{X^{k}+Y^{k}}{2} -\frac{\nabla f (Y^{k})}{2\beta}$.
\end{theorem}
\begin{proof}
	The proof follows the same argument of \cite[Theorem 3.1]{IRSVM_LuZhaoSong_2014}.
\end{proof}

With the help of \Cref{Theo:SubproblemSolve} and \Cref{Def:simultaneousSVD} of simultaneous ordered SVD, we establish the following result whose proof follows the similar argument of \cite[Proposition 2]{ge2022fast}.

\begin{proposition}\label{prop_opt}
	Consider \eqref{eq:NewIter}. Suppose $\bm{w}^{k} \in \mathbb{R}_{\uparrow}^{m} \cap \mathbb{R}^{m}_{++}$. Then there exist $\bm{\xi}^{k+1} \in \partial \vert \bm{\sigma}(X^{k+1}) \vert \subset \mathbb{R}^{m}$  such that
	\begin{equation}\label{eq:optimalitySub}
		U^{k+1}\text{diag}\left( \frac{\lambda}{2\beta}\bm{w}^{k} \circ \bm{\xi}^{k+1} \right) {V^{k+1}}^{\top} \in \partial\left\lbrace \sum\limits_{i=1}^{m}w_{i}^{k}\sigma_{i} (X^{k+1}) \right\rbrace, 
	\end{equation}
	where $X^{k+1}$ and $\frac{X^{k}+Y^{k}}{2}- \frac{\nabla f(Y^{k})}{2\beta}$ have a simultaneous ordered SVD. 
\end{proposition}
\begin{proof}
	By \Cref{theo_Fermat} and \Cref{Theo:SubproblemSolve}, we have 
	\begin{equation}\label{eq:optsubconditon}
		\bm{0} \in \beta\left( X^{k+1} -\left(\frac{X^{k}+Y^{k}}{2}-\frac{\nabla f(Y^k)}{2\beta} \right)  \right) + \frac{\lambda}{2}\partial\left(\sum_{i=1}^{m} w_{i}^{k} \sigma_{i}(X^{k+1}) \right).
	\end{equation}
	Note also that matrices $X^{k+1} $ and $\frac{X^{k}+Y^k}{2} -\frac{1}{2\beta}\nabla f (X^{k}) $ have the simultaneous ordered SVD. From \eqref{eq:optsubconditon}, we know there exists $\hat{\xi}^{k+1} \in \partial\left(\sum_{i=1}^{m} w_{i}^{k} \sigma_{i}(X^{k+1}) \right)$ such that 
	\begin{equation}
		\begin{aligned}
			\hat{\xi}^{k+1} &= \frac{2\beta}{\lambda}\left(\left(\frac{X^{k}+Y^{k}}{2}-\frac{\nabla f(Y^k)}{2\beta} \right) -  X^{k+1}\right)\\
			&= \frac{2\beta}{\lambda}\left(U^{k+1}\textrm{diag}(\Sigma^{k+1}){V^{k+1}}^{\top} - U^{k+1} \text{diag}\left(\left(\Sigma^{k+1}-\frac{\lambda \bm{w}^{k}}{2\beta}\right)_{+}\right) {V^{k+1}}^{\top}\right)\\
			&= \frac{2\beta}{\lambda}\left(U^{k+1}\textrm{diag}\left( \Sigma^{k+1}-\left( \Sigma^{k+1}-\frac{\lambda \bm{w}^{k}}{2\beta}\right)_{+}\right)  {V^{k+1}}^{\top}\right).
		\end{aligned} 
	\end{equation}
	If $\sigma_{i}(X^{k+1}) = \left( \Sigma^{k+1}_i-\frac{\lambda w_i^{k}}{2\beta}\right)_{+} >0$ for $i\in [m]$, we know $\partial (\sigma_{i}(X^{k+1})) =1$ and $\Sigma^{k+1}-\left( \Sigma^{k+1}-\frac{\lambda \bm{w}^{k}}{2\beta}\right)_{+} = \frac{\lambda \bm{w}^{k}}{2\beta} $. Then, it follows from for any $i \in [m]$ such that $\xi_{i}^{k+1} = 1$ that $\frac{\lambda\bm{w}^{k} }{2\beta}\circ \bm{\xi}^{k+1} = \frac{\lambda\bm{w}^{k}}{2\beta} \in \frac{\lambda\bm{w}^{k}}{2\beta}\circ \partial (\sigma_{i}(X^{k+1})) $. If $\sigma_{i}(X^{k+1}) = \left( \Sigma^{k+1}_i-\frac{\lambda w_i^{k}}{2\beta}\right)_{+} = 0$ for $i\in [m]$, we know $\partial (\sigma_{i}(X^{k+1})) = [-1,1]$ and $\Sigma^{k+1}_{i}-\left( \Sigma^{k+1}_{i}-\frac{\lambda \bm{w}^{k}_{i}}{2\beta}\right)_{+} =  \Sigma^{k+1}_{i} \in [0,\frac{\lambda \bm{w}_{i}^{k}}{2\beta}]$. Then, it follows from for any $i \in [m]$ such that $\xi_{i}^{k+1} = \frac{2\beta\Sigma_{i}^{k+1}}{\lambda w_{i}^k} \in [0,1]\subset[-1,1]$ that $\frac{\lambda\bm{w}^{k} }{2\beta}\circ \bm{\xi}^{k+1} = \Sigma_{i}^{k+1}\in\frac{\lambda\bm{w}^{k} }{2\beta} \partial(\sigma_{i}(X^{k+1}))$. Therefore, the proof is completed.
\end{proof}

Since $X^{k+1}$ is a global minimum to \eqref{eq:NewIter} by \Cref{Theo:SubproblemSolve}, it follows from \Cref{theo_Fermat} and \Cref{prop_opt} that  there exist $\bm{\xi}^{k+1}\in\partial|\bm{\sigma}(X^{k+1})| $ such that 
\begin{equation}\label{KKT_acc_opt}
	\bm{0} = \nabla f (Y^{k}) + \beta (X^{k+1}-Y^{k}) + \beta (X^{k+1} - X^{k}) + \lambda U^{k+1}\mathrm{diag} \left(\bm{w}^{k}\circ\bm{\xi}^{k+1}\right){V^{k+1}}^{\top}. 
\end{equation}

\subsection{An Adaptive Updating Strategy for Perturbation \texorpdfstring{\boldmath{$\epsilon$}}{e}} 
A key component of our proposed algorithmic framework is the updating strategy for the perturbation $\bm{\epsilon}$, which controls how the perturbation evolves during optimization and is critical for analyzing the behavior of our proposed algorithm. 
The updating strategy should be designed such that 
we can manipulate the values of $\epsilon_{i}, \forall i \in [m]$ to maintain the ascending order of $\{ w_1, \ldots, w_m\}$ during  iteration. This ensures that a global optimal solution to subproblem \eqref{eq:NewIter} can be obtained according to \Cref{Theo:SubproblemSolve}. Our proposed updating strategy is presented in \Cref{Algo:update_eps}. 

\begin{algorithm}[htbp] 
\caption{Update perturbation $\bm{\epsilon}$.}
\label{Algo:update_eps} 
\begin{algorithmic}[1]
	\REQUIRE{ $\mu\in(0,1)$}.
	\IF{$\Ical(X^{k+1}) \subset \Ical(X^{k})$}\label{update_eps:start}
	\STATE  $\epsilon_{i}^{k+1} = \mu \epsilon_i^k$,   $ \forall i \in  \Ical^{k+1}.$ \label{update_eps:NonzerosFrac}
	\STATE  Set $\tau_1 = \sigma_{|\mathcal{I}^{k+1}|}^{k+1} + \epsilon_{|\mathcal{I}^{k+1}|}^{k+1} $ and $\tau_2  = \epsilon_{|\mathcal{I}^{k+1}|+1}^{k}$. \label{update_eps:1stthreshold}
	\STATE  $\epsilon_{i}^{k+1} = 
	\begin{cases} 
		\epsilon_i^k , & \text{ if }  \tau_1 \geq 	\tau_2 , \\
		\min(\epsilon_i^k,\mu \tau_1), & \text{ otherwise, }
	\end{cases}$ $\forall i \in \mathcal{I}(X^{k})\backslash \mathcal{I}(X^{k+1})$.\label{update_eps:ZeroNonzeroIntersection}
	\STATE Set $\tau_3  = \epsilon_{|\mathcal{I}(X^{k})|}^{k+1}$
	\STATE $\epsilon_{i}^{k+1} = \min\{\epsilon_{i}^{k},\tau_3 \}$,  $\forall i \in \mathcal{Z}(X^{k})$. \label{updateeps.z}
	\ENDIF
	\IF{$\Ical(X^k) \subset \Ical(X^{k+1})$} 
	\STATE	$\epsilon_{i}^{k+1} = \mu \epsilon_i^k$,   $ i \in \Ical(X^k).$
	\STATE  Set $\tau_3  = \epsilon_{|\mathcal{I}(X^{k})|}^{k}$
	\STATE 	$\epsilon_{i}^{k+1} =\mu \min \{\epsilon_{i}^{k}, \tau_3 \}$, $\forall i \in \mathcal{I}(X^{k+1})\backslash\mathcal{I}(X^{k})$ \label{update_eps:Zeros}
	\STATE 	Set $\tau_1 = \sigma_{|\mathcal{I}(X^{k+1})|}^{k+1} + \epsilon_{|\mathcal{I}(X^{k+1})|}^{k+1} $ and $\tau_2  = \epsilon_{|\mathcal{I}(X^{k+1})|+1}^{k}$.
	\STATE $\epsilon_{i}^{k+1} = 
	\begin{cases} 
		\epsilon_i^k , & \text{ if }  \tau_1 \geq \tau_2 , \\
		\min(\epsilon_i^k,\mu \tau_1), & \text{ otherwise, }
	\end{cases}$ $\forall i \in \mathcal{Z}(X^{k+1})$.
	\ENDIF
	\IF{$\Ical(X^k) = \Ical(X^{k+1})$}
	\STATE $\epsilon_{i}^{k+1} = \mu \epsilon_i^k, \ \forall  i \in \mathcal{I}(X^{k+1}).$\label{update_eps:nonzerodemishes}
	\STATE Set $\tau_1 = \sigma_{|\mathcal{I}(X^{k+1})|}^{k+1} + \epsilon_{|\mathcal{I}(X^{k+1})|}^{k+1} $ and $\tau_2  = \epsilon_{|\mathcal{I}(X^{k+1})|+1}^{k}$.  
	\STATE $\epsilon_{i}^{k+1} = 
	\begin{cases} \epsilon_i^k,              & \text{ if } \tau_1 \geq \tau_2 , \\
		\min(\epsilon_i^k,\mu \tau_1 ),  & \text{ otherwise,} 
	\end{cases}$ $\forall i\in \mathcal{Z}(X^{k+1})$. \label{update_eps:zerofixed}
	\ENDIF
\end{algorithmic}
\end{algorithm}

Assume the initial  $\epsilon_{i}^0, i \in[m]$ are in descending order.   \Cref{Algo:update_eps} includes three cases.  Our focus is mainly on providing detailed explanations for the first case, as other cases follow similar arguments. \textbf{Case 1: $\Ical(X^{k+1}) \subset \Ical(X^{k})$ holds true  in Line \ref{update_eps:start}.} This case corresponds to a situation in which $X^{k+1}$ have more zero singular values than $X^k$, meaning  $\vert \mathcal{I}(X^{k+1})\vert < \vert \mathcal{I}(X^{k})\vert$, or, equivalently, $\Rank(X^{k+1}) < \Rank(X^k)$. Notice that all elements in $\bm{\sigma}(X^{k+1})$ are organized naturally in descending order. Our goal is to maintain the descending order of $ \sigma_{i}^{k+1} + \epsilon_{i}^{k+1}, i \in [m]$ (or, equivalently, the ascending order of $w_i^{k+1}$, $i\in[m]$).
To achieve this, we first decrease $\epsilon_{i}^{k+1}, \forall i \in \mathcal{I}(X^{k+1})$ by a fraction (Line \ref{update_eps:NonzerosFrac}), so that $\sigma_{i}^{k+1} + \epsilon_{i}^{k+1}, i \in \mathcal{I}(X^{k+1})$ are descending after the update. Let $\tau_1$ be their smallest value.  Line \ref{update_eps:1stthreshold}-\ref{update_eps:ZeroNonzeroIntersection} handle the update of $\epsilon_{i}^{k+1}, \forall i \in \mathcal{I}(X^{k})\backslash \mathcal{I}(X^{k+1})$. 
Let  $\tau_2 $ be the largest value of $\epsilon_{i}^k, \forall i \in \mathcal{I}(X^{k})\backslash \mathcal{I}(X^{k+1})$ (they are in descending order).  If $\tau_1 \ge \tau_2$,  there is no need to reduce $\epsilon_{i}^k, \forall i \in \mathcal{I}(X^{k})\backslash \mathcal{I}(X^{k+1})$, since $\sigma_{i}^{k+1} + \epsilon_{i}^{k+1}, i \in \mathcal{I}^k = \mathcal{I}(X^{k+1}) \cup (\mathcal{I}(X^{k})\backslash \mathcal{I}(X^{k+1}))$ are now descending. Otherwise,  
$\epsilon_{i}^k, \forall i \in \mathcal{I}(X^{k})\backslash \mathcal{I}(X^{k+1})$ is set to be a fraction of $\tau_1$ to maintain this order.   As for $i \in \mathcal{Z}(X^{k})$, letting $\tau_3$ be the smallest value of  $\sigma_{i}^{k+1} + \epsilon_{i}^{k+1}, i \in \mathcal{I}^k$,  we then use threshold $\tau_3 $ to trim $i \in \mathcal{Z}(X^{k})$ (Line \ref{updateeps.z}). 
After the update, $\sigma_{i}^{k+1} + \epsilon_{i}^{k+1}, i \in [m]$ maintain the desired non-decreasing order.
\section{Convergence Analysis}
We make heavy use of the following auxiliary function in our analysis, which is useful in establishing the convergence properties of the proposed algorithm. 
\begin{equation}\label{pro_prox}
	H (X,Y,\epsilon ) = f(X) + \frac{\beta}{2}\|X-Y\|_{F}^{2} +   \lambda\sum_{i=1}^{m} (\sigma_{i} (X)+\epsilon_{i})^{p}.
\end{equation}
The following lemma indicates that $H (X,Y,\epsilon )$ is monotonically nonincreasing. 

\begin{lemma}[Sufficient decrease property of EIRNRI] \label{lem_Dec_acc}
	Suppose Assumptions \ref{ASM_fLip}-\ref{ASM_levelBoundness} are satisfied. Let $\{X^{k}\} $ and $\{Y^{k}\}$ be the sequences generated by \Cref{alg_acc}. Then the following statements hold.
	\begin{itemize}
		\item $\left\{H (X^{k},X^{k-1},\bm{\epsilon}^{k})\right\} $ is monotonically nonincreasing and  $\lim\limits_{k\to +\infty} H(X^{k},X^{k-1},\bm{\epsilon}^{k})$ exists. Indeed, we have for each $ k \in \mathbb{N}$ that
		\begin{equation}\label{eq:usefulIneq0}
			H (X^{k},X^{k-1},\bm{\epsilon}^{k}) - H (X^{k+1},X^{k},\bm{\epsilon}^{k+1})\ge
			C\|X^{k}-X^{k-1}\|_{F}^{2}
		\end{equation}
		with constant $C =  \frac{\beta}{2} (1-\frac{3L_{f}+\beta}{\beta}\bar{\alpha}^{2}) > 0.$
		
		\item $\sum_{k=0}^{\infty} \|X^{k+1}-X^{k}\|_{F}^{2} < +\infty$, implying $\lim\limits_{k\to+\infty} \|X^{k+1}-X^{k}\|_{F}=0$. In addition, $\lim\limits_{k\to+\infty} \max\{\|Y^{k}-X^{k}\|_{F}, \|Y^{k}-X^{k+1}\|_{F}\}=0$. 
		
		\item The sequences $\{X^{k}\}$ and $\{Y^{k}\}$ are bounded. As a result, there exists constant $C_1$ such that for any $i \in [m]$
		\begin{equation*}
			\max\limits_{i} \, \sigma_{i} \left(\frac{X^{k}+Y^k}{2} -\frac{1}{2\beta}\nabla f (Y^{k})\right)\le C_1,\  \forall k \in \mathbb{N}.
		\end{equation*}
	\end{itemize}
\end{lemma}
\begin{proof}
	
	(i) From \eqref{eq:NewIter}, we know that $X^{k+1} \in \argmin_{X \in \mathbb{R}^{m\times n}} L(X;X^{k},Y^{k},\bm{\epsilon}^{k})$, and it follows from $ L(X^{k+1};X^{k},Y^{k},\bm{\epsilon}^{k}) \leq L(X^{k};X^{k},Y^{k},\bm{\epsilon}^{k})$ that
	\begin{equation} \label{weighted_sigma_k}
		\begin{aligned}
			&\ \langle X^{k+1},\nabla{f 	(Y^{k})}\rangle+\frac{\beta}{2}\|X^{k+1}-Y^{k}\|+ \frac{\beta}{2}\|X^{k+1}-X^{k}\|_{F}^{2} +\lambda\sum\limits_{i=1}^{m}w_{i}^{k}\sigma_i(X^{k+1}) \\
			\le&\ \langle X^{k},\nabla{f 	(Y^{k})}\rangle+\frac{\beta}{2}\|X^{k}-Y^{k}\|+\lambda\sum\limits_{i=1}^{m}w_{i}^{k}\sigma_i(X^k).
		\end{aligned}
	\end{equation}
	By rearranging \eqref{weighted_sigma_k} and adding a positive term on both sides, it holds that 
	\begin{equation}\label{eq:usefulIneq}
		\begin{aligned}
			&\ \langle X^{k+1},\nabla{f 	(Y^{k})}\rangle+\frac{\beta}{2}\|X^{k+1}-Y^{k}\| +\lambda\sum\limits_{i=1}^{m}w_{i}^{k}(\sigma_i(X^{k+1})-\sigma_i(X^k)) + \lambda\sum\limits_{i=1}^{m}(\sigma_i(X^k)+\epsilon_{i}^k)^p\\
			\le&\ \langle X^{k},\nabla{f 	(Y^{k})}\rangle+\frac{\beta}{2}\|X^{k}-Y^{k}\| -  \frac{\beta}{2}\|X^{k+1}-X^{k}\|_{F}^{2} + \lambda\sum\limits_{i=1}^{m}(\sigma_i(X^k)+\epsilon_{i}^k)^p.
		\end{aligned}
	\end{equation}  
	Denote $\phi(X;Y) = f(Y) + \langle \nabla f(Y), X - Y\rangle$ for notional convenience. It leads us to 
	\begin{equation}\label{acc_dec_Strong}
		\begin{aligned}
			& \ F(X^{k+1};\bm{\epsilon}^{k+1})\\
			\overset{(a)}{\leq} & \ \phi(X^{k+1};Y^{k}) + \frac{L_{f}}{2}\|X^{k+1}-Y^{k}\|_{F}^{2} +\lambda\sum\limits_{i=1}^{m} (\sigma_i(X^{k+1}) +\epsilon_{i}^{k})^{p} \\
			\overset{(b)}{\leq} &\ \phi(X^{k+1};Y^{k}) + \frac{\beta}{2} \|X^{k+1}-Y^{k}\|_{F}^{2}+ \lambda\sum\limits_{i=1}^{m}(\sigma_i(X^k)+\epsilon_{i}^k)^p \\
			&\  +\lambda\sum\limits_{i=1}^{m} w_{i}^{k}(\sigma_i(X^{k+1})-\sigma_i(X^{k})) \\
			\overset{(c)}{\leq} &\ \phi(X^{k};Y^{k})+\frac{\beta}{2}\|X^{k}-Y^{k}\|_{F}^{2} - \frac{\beta}{2}\|X^{k+1}-X^{k}\|_{F}^{2} +\lambda\sum\limits_{i=1}^{m} (\sigma_i(X^k)+\epsilon_{i}^{k})^{p}  \\
			\overset{(d)}{\leq} &\ f(X^{k})+\langle \nabla{f (Y^{k})-\nabla f (X^{k})},X^{k}-Y^{k}\rangle+\frac{L_{f}+\beta}{2}\|X^{k}-Y^{k}\|_{F}^{2}\\
			&\ - \frac{\beta}{2}\|X^{k+1}-X^{k}\|_{F}^{2} + \lambda\sum\limits_{i=1}^{m} (\sigma_i(X^k)+\epsilon_{i}^{k})^{p} \\
			\overset{(e)}{\leq} &\ F(X^{k},\bm{\epsilon}^{k})+\frac{3L_{f}+\beta}{2}\|X^{k}-Y^{k}\|_{F}^{2}- \frac{\beta}{2}\|X^{k+1}-X^{k}\|_{F}^{2},
		\end{aligned}
	\end{equation}
	where inequalities $(a)$ follows from the result in \cite[Lemma 1.2.3]{nesterov2003introductory} under \Cref{ASM_fLip} and leverages the monotonicity of $(\cdot)^{p}$ over $\mathbb{R}_{+}$ resulting from the nonincreasing property of $\sigma_i(X^{k})+\epsilon_i^k, \forall i\in [m],  k \in \mathbb{N}$ by \Cref{Algo:update_eps}, $(b)$ is true because $\beta > L_f$ and the concavity of $(\cdot)^{p}$ over $\mathbb{R}_{+}$, $(c)$ makes use of \eqref{eq:usefulIneq}, $(d)$ again follows from the result in \cite[Lemma 1.2.3]{nesterov2003introductory} under \Cref{ASM_fLip}, and $(e)$ is by the Cauchy-Schwarz inequality and hence immediately a consequence of \Cref{ASM_fLip}.
	
	Rearranging \eqref{acc_dec_Strong}, together with \eqref{eq: extrapolatedPoint}, yields   
	\begin{equation}\label{eq:UsefulIneq2}
		F (X^{k};\bm{\epsilon}^{k}) - F (X^{k+1};\bm{\epsilon}^{k+1}) \ge  \frac{\beta}{2}\|X^{k}-X^{k+1}\|_{F}^{2} - \frac{3L_{f}+\beta}{2}\alpha_{k}^{2}\|X^{k}-X^{k-1}\|_{F}^{2}.
	\end{equation}
	This further implies 
	\begin{equation}\label{acc_dec_H}
		\begin{aligned}
			&\ H (X^{k},X^{k-1},\bm{\epsilon}^{k})-H (X^{k+1},X^{k},\bm{\epsilon}^{k+1}) \\ 
			=&\ F (X^{k};\bm{\epsilon}^{k})+\frac{\beta}{2}\|X^{k}-X^{k-1}\|_{F}^{2} - \left[F (X^{k+1};\bm{\epsilon}^{k+1})+\frac{\beta}{2}\|X^{k+1}-X^{k}\|_{F}^{2}\right] \\
			\overset{(a)}{\ge}&\ \frac{\beta}{2} \left(1-\alpha_{k}^{2}\frac{3L_{f}+\beta}{\beta}\right)\|X^{k}-X^{k-1}\|_{F}^{2} \\
			\overset{(b)}{\ge}&\ \frac{\beta}{2} \left(1-\bar{\alpha}^{2}\frac{3L_{f}+\beta}{\beta}\right)\|X^{k}-X^{k-1}\|_{F}^{2}\geq 0,
	\end{aligned}\end{equation} 
	where inequality $(a)$ holds by \eqref{eq:UsefulIneq2} and $(b)$ is true thanks to $\alpha_{k}\in[0,\bar{\alpha}], \forall k \in \mathbb{N}$ with $\bar{\alpha} \in (0,\sqrt{\frac{\beta}{3L_{f}+\beta}})$ imposed in \eqref{alpha_update}.
	Consequently,  $\{H (X^{k},X^{k-1},\bm{\epsilon}^{k})\} $ is monotonically nonincreasing. Moreover, by \Cref{ASM_levelBoundness}, we know $\min_{X \in \mathbb{R}^{m\times n}} F(X,\bm{\epsilon}) > \underline{F} >-\infty$, implying $\{H (X^{k},X^{k-1},\bm{\epsilon}^{k})\}$ is bounded from below, and hence, $\lim\limits_{k\to +\infty} H(X^{k},X^{k-1},\bm{\epsilon}^{k})$ exists.
	This proves statement (i).

	(ii) Summing both sides of \eqref{acc_dec_H} over $k = 0, \ldots, t $, we obtain
	\begin{equation}
		\begin{aligned}
			\frac{\beta}{2} 	\left(1-\bar{\alpha}^{2}\frac{3L_{f}+\beta}{\beta}\right)\sum\limits_{k=0}^{t} \|X^k-X^{k-1}\|_{F}^{2} &\le\ H(X^{0},X^{-1},\bm{\epsilon}^{0}) - H 	(X^{t+1},X^{t},\bm{\epsilon}^{t+1}) \\
			&\le\ F (X^{0};\bm{\epsilon}^{0}) -   F(X^{t+1};\bm{\epsilon}^{t+1}) 
			<\ +\infty,
		\end{aligned}
	\end{equation}
	Let $t\to+\infty$, and it follows from \Cref{ASM_levelBoundness} and $\bar{\alpha} \in (0,\sqrt{\frac{\beta}{3L_{f}+\beta}})$ that 
	$\lim\limits_{k\to+\infty}\|X^{k+1}-X^{k}\|_{F}=0 $. Furthermore,  
	\begin{equation}
		\begin{aligned}
			\lim\limits_{k\to+\infty}\|Y^{k}-X^{k}\|_{F} &= \lim\limits_{k\to+\infty}\alpha_{k}\| (X^{k}-X^{k-1}) \|_{F} = 0, \textrm{ and }\\
			\lim\limits_{k\to+\infty}\|Y^{k}-X^{k+1}\|_{F} &=  \lim\limits_{k\to+\infty}\Vert(X^{k}-X^{k+1})+\alpha_{k} (X^{k}-X^{k-1})\Vert_{F} = 0,
		\end{aligned}
	\end{equation}
	as desired. The proof of statement (ii) is completed.
	
	(iii) For each $k \in \mathbb{N}$, we derive that
	\begin{equation} \label{ineq_MM_chain}
		F (X^{k}) \le F (X^{k};\bm{\epsilon}^{k}) \le H (X^{k},X^{k-1},\bm{\epsilon}^{k}) \overset{(a)}{\leq} H(X^{0},X^{-1},\bm{\epsilon}^{0}) = F (X^{0};\bm{\epsilon}^{0}) \overset{(b)}{<} +\infty,
	\end{equation}
	where inequality $(a)$ holds by \eqref{acc_dec_H} and \Cref{thm.update}(i), and $(b)$ is true due to \Cref{ASM_levelBoundness}. We hence deduce from \eqref{ineq_MM_chain} that $\{X^{k}\}$ is bounded. This, together with the boundedness of $\alpha_k, \forall k\in \mathbb{N}$, shows the boundedness of $\{Y^{k}\}$. Therefore, it follows from Assumptions \ref{ASM_fLip}-\ref{ASM_levelBoundness} that there exists $C_1 > 0 $ such that for any $i \in [m]$
	\begin{equation*}
		\max\limits_{i} \, \sigma_{i} \left(\frac{X^{k}+Y^k}{2} -\frac{1}{2\beta}\nabla f (Y^{k})\right)\le C_1,\  \forall k \in \mathbb{N}.
	\end{equation*}
	This completes proofs of all statements. 
\end{proof}

\subsection{Local Properties: Stable Support and Adaptively Reweighting}
The following proposition asserts that the support of the index set of singular vectors remains unchanged after a finite number of iterations.
\begin{proposition}\label{THM_away_zero}
	Let Assumptions \ref{ASM_fLip}-\ref{ASM_levelBoundness} hold. Suppose the sequence $\{X^{k}\} $ is generated by \Cref{alg_acc}. Then there exists constant $C_1>0$ defined in \Cref{lem_Dec_acc} and $\hat{k} \in\mathbb{N}$ such that the following statements hold.
	\begin{itemize}
		\item If $w(\sigma_{i} ({X^{\hat{k}}}),\epsilon_{i}^{\hat{k}}) > \frac{2\beta C_1}{\lambda} $ for $\hat{k}\in\mathbb{N}$, then $\sigma_i(X^{k}) \equiv 0 $ for all $k>\hat{k} $.
		
		\item The index set sets $\mathcal{I}(X^k)$ and $\mathcal{Z}(X^{k})$ remain unchanged for all $k > \hat{k}$. Therefore, there exists index sets $\Ical^*$ and $\mathcal{Z}^*$ such that $ \mathcal{I}(X^k) = \Ical^*$ and $\mathcal{Z}(X^{k})=\mathcal{Z}^*$ for sufficiently large $k$. 
		
		\item For any $k > \hat{k}$, $\sigma_i^k$ is strictly bounded away from $0$ for any $i \in \mathcal{I}(X^k)$. Indeed, it holds that $\sigma_i^k> \left(\frac{\lambda p}{2\beta C_1}\right)^{\frac{1}{1-p}}-\epsilon_{i}^{k}> 0 , \forall i\in\mathcal{I} (X^{k})$, implying $\liminf\limits_{k\to+\infty} \sigma_i^k > \left(\frac{\lambda p}{2\beta C_1}\right)^{\frac{1}{1-p}}, \forall i\in\mathcal{I}^*$.   
	\end{itemize}     
\end{proposition}
\begin{proof}
	Recall \eqref{eq:NewIter} and by \Cref{prop_opt}, we know that matrices $X^{k+1} $ and $\frac{X^{k}+Y^k}{2} -\frac{1}{2\beta}\nabla f (X^{k}) $ have the simultaneous ordered SVD, and 
	\begin{equation*}
		\bm{0} \in \beta\left( X^{k+1} -\left(\frac{X^{k}+Y^{k}}{2}-\frac{\nabla f(Y^k)}{2\beta} \right)  \right) + \frac{\lambda}{2}\partial\left(\sum_{i=1}^{m} w_{i}^{k} \sigma_{i}(X^{k+1}) \right),
	\end{equation*} 
	which implies  
	\begin{equation}\label{KKT_sgv}
		\bm{0} =  \bm{\sigma}(X^{k+1}) - \bm{\sigma} \left(\frac{X^{k}+Y^k}{2} -\frac{1}{2\beta}\nabla f (Y^{k})\right) + \frac{\lambda}{2\beta}{\bm w}^{k}\circ\bm{\xi}^{k+1}
	\end{equation}
	with ${\xi}_{i}^{k+1} \in[0,1] $. Then, we have for each $i \in [m]$ that 
	\begin{equation}\label{eq:nonnegativeSVDvalue}
		\sigma_i(X^{k+1})  = \left[ \sigma_{i} \left(\frac{X^{k}+Y^k}{2} -\frac{1}{2\beta}\nabla f (Y^{k})\right) - \frac{\lambda}{2\beta}w_{i}^{k}{\xi}_{i}^{k+1} \right]_{+}.
	\end{equation}

	(i) Suppose that there exists $\hat k \in \mathbb{N}$ such that $w_{i}^{\hat k} \ge \frac{2\beta C_1}{\lambda} $ for $i\in[m]$. 
	By \Cref{prop_opt} and from \eqref{eq:nonnegativeSVDvalue}, we know that $\sigma_i(X^{\hat k+1}) = 0$. Then, $\sigma_{i}(X^{\hat{k}+1})+ \epsilon_{i}^{\hat{k}+1} \le \sigma_{i}(X^{\hat{k}}) + \epsilon_{i}^{\hat{k}}$ by \Cref{Algo:update_eps} and monotonicity of $ (\cdot)^{p-1} $ indicate 
	$w_{i}^{\hat{k}+1} \geq w_{i}^{\hat{k}} >  \frac{2\beta C_1}{\lambda}$.  
	Therefore, we have $\sigma_i^{\hat k+2} = 0$. By induction we know that $\sigma_i^{k} \equiv 0 $ for any $k>\hat{k} $. This completes the proof of statement (i).
	
	(ii)       
	We prove this by contradiction. Suppose this statement is not true. Then there exist $i\in[m] $ and $k \in \mathbb{N}$ such that $\sigma_i(X^k) $ takes zero and nonzero value both for infinite times. We know that there are two subsequences $\mathcal{S}_{1}\cup\mathcal{S}_{2}=\mathbb{N} $ such that $|\mathcal{S}_{1}|=+\infty $, $|\mathcal{S}_{2}|=+\infty $ and that 
	\begin{equation*}
		\sigma_i(X^{k}) = 0 ,\; \forall k\in\mathcal{S}_{1} \text{ and } 
		\sigma_i(X^{k}) > 0 ,\; \forall k\in\mathcal{S}_{2}.
	\end{equation*}
	Hence,  there exists subsequence $\mathcal{S}_3 \subset \mathcal{S}_2$ such that $\vert \mathcal{S}_3\vert = +\infty$ and $i \in \mathcal{Z}^{X^{k}} \cap \mathcal{I}(X^{k+1})$ for any $k \in \mathcal{S}_3$. In other words, $\sigma_{i}(X^{k}) = 0$ and $\sigma_{i}(X^{k+1}) \neq 0$ for any $k \in \mathcal{S}_3$. Thus, $\mathcal{I}^{k} \subset \mathcal{I}^{k+1}$, $\forall k \in \mathcal{S}_3$. It then follows from \Cref{Algo:update_eps} that $\lim\limits_{k\to+\infty}\epsilon_{i}^{k}=0$ for $k \in \mathcal{S}_3$ owing to $|\mathcal{S}_{3}|=+\infty$. 
	Hence there exists $\hat{k}\in\mathcal{S}_1$ such that 
	\begin{equation*}
		w_{i}^{\hat{k}}=w(\sigma_i^{\hat{k}} ,\epsilon_i^{\hat{k}}) = p \left(\sigma_i (X^{\hat{k}})+\epsilon_i^{\hat{k}} \right)^{p-1} = p (\epsilon_i^{\hat{k}})^{p-1} \ge \frac{2\beta C_1}{\lambda}. 
	\end{equation*}
	This indicates that $\sigma_i (X^{k})=0 $ for any $k>\hat{k}$ by statement (i), implying $\{\hat{k}+1,\hat{k}+2,\hat{k}+3,\cdots\}\subset\mathcal{S}_{1} $ and $|\mathcal{S}_{2}|$ is finite. This contradicts $|\mathcal{S}_{2}|=+\infty $. Consequently, statement (ii) holds. 
	
	(iii) We know from statement (i)  that if $w_{i}^{k}\leq\frac{2\beta C_1}{\lambda}$, $i\in\Ical(X^k)$ occurs, it then follows that  
	\begin{equation*}
		\sigma_i(X^k)\geq \left(\frac{\lambda p}{2\beta C_1}\right)^{\frac{1}{1-p}}-\epsilon_{i}^{k} > 0,\ \forall i\in\mathcal{I}(X^k).
	\end{equation*}
	This completes proofs of all statements.
\end{proof} 

We shall further demonstrate the properties for the updating strategy of $\bm{\epsilon}$. We specifically show that after some $k$, $\epsilon_{i}, \forall i \in \mathcal{I}(X^{k})$ diminish while $\epsilon_{i}, \forall i \in \mathcal{Z}(X^{k})$ are fixed as constants. Hence, our proposed \Cref{alg_acc} locally behaves like minimizing a smooth problem in a low-dimensional manifold.

\begin{theorem} \label{thm.update}
	Suppose \Cref{ASM_fLip}-\ref{ASM_levelBoundness} hold true. Let $\{X^{k}\} $ and  $\{\bm{\epsilon}^{k}\}$ be the sequences generated by \Cref{alg_acc} and $\Ical^*$ and $\mathcal{Z}^*$ be defined in Proposition \ref{THM_away_zero}(ii). Then the following asserts hold. 
	\begin{itemize}
		\item The perturbations $\{ \epsilon_i^k, \forall i \in \Ical(X^k)\}$ and perturbed singular values $\{\sigma_i(X^{k})+\epsilon_i^k, \forall i\in[m]\}$ are all in strictly descending order for $k \in \mathbb{N}$, while $\{ \epsilon_i^k, \forall i\in \mathcal{Z}(X^k)\}$ are in non-increasing order for $k \in \mathbb{N}$. Consequently, for all $ k \in \mathbb{N}$, the non-strictly ascending order constraint on the non-negative weights in \eqref{eq:AscendingWeights}  can be automatically satisfied.
		
		\item There exist $\hat{k} \in \mathbb{N}$ such that for all $k \geq \hat{k}$, the update $\epsilon_{i}^{k+1} = \mu \epsilon_{i}^{k}$ with $\mu \in (0,1)$, $\forall i \in \Ical^{*}$   will always be triggered. Consequently, the sequence $\{\epsilon_i^k\}, \forall i \in \Ical^*$ converges monotonically to $0$, i.e., $\epsilon_i^k \to 0$ as $k \to +\infty$ for all $i \in \Ical^*$.
		
		\item There exists $\hat{k} \in \mathbb{N}$ such that for all $k \geq \hat{k}$, the update of
		$\epsilon_i^{k}, \forall i\in\mathcal{Z}(X^{k})$ will never be triggered. That is, $\epsilon_i^{k} \equiv \epsilon_i^{\hat{k}}$ after some $\hat{k}$, $\forall i\in\mathcal{Z}^*$. Consequently, the sequence $\{\epsilon_i^k\}$, $\forall i \in \mathcal{Z}^*$ converges to fixed positive constants for all sufficiently large $k$.
	\end{itemize}
\end{theorem}
\begin{proof}
	
	(i)  We prove this statement by induction. Indeed, by the setting of $\bm{\epsilon}^{0}$ in \Cref{alg_acc} and the property that the singular values are naturally sorted in descending order, the statement is vacuously true at $k=0$. Suppose now this is also true at the $k$th iteration. Without loss of generality, we only have to prove the statement (i) in which $\mathcal{I}(X^{k+1}) \subset \mathcal{I}(X^{k})$ holds in \Cref{Algo:update_eps}, and the proof of other cases follows the similar spirits and arguments.
	
	Consider \Cref{Algo:update_eps}. Line \ref{update_eps:NonzerosFrac} and Line \ref{update_eps:ZeroNonzeroIntersection} indicate that $\{\epsilon_i^{k+1}, \forall i\in  \Ical(X^{k})\}$ is in descending order by the descending nature of  
	$\{\epsilon_i^k, \forall i\in \Ical(X^k)\}$ and $\mu \in (0,1)$.  Line \ref{update_eps:Zeros} guarantees 
	$\{\epsilon_i^{k+1}, i\in \mathcal{Z}(X^{k})\}$ is non-increasing, since $\{\epsilon_{i}^{k},\forall i \in \mathcal{Z}(X^{k})\}$ are in non-increasing order. This, together with the monotonicity of $(\cdot)^{p-1}$, ensures the satisfaction of \eqref{eq:AscendingWeights}. This finishes the proof of statement (i).

	(ii)  This statement holds true by Proposition \ref{THM_away_zero}(ii) and Line \ref{update_eps:nonzerodemishes} of \Cref{Algo:update_eps}. 
	
	(iii) Seeking a contradiction, suppose the update of $\epsilon_{i}^{k}, \forall i \in \mathcal{Z}(X^{k})$ for $k > \hat{k}$ is triggered for infinite many times. By Proposition \ref{THM_away_zero}(ii), we note that $\epsilon_i^{k+1} \le \mu\tau_1 , i\in\mathcal{Z}(X^{k+1})$ whenever 
	it is reduced by Line \ref{update_eps:zerofixed} in \Cref{Algo:update_eps}.  
	If the update is triggered for infinite many times,  
	then $\tau_2  \le \tau_1 $ is always satisfied for any $k > \hat{k}$, which contradicts that $\tau_1 $ is strictly bounded below from $0$ after some $k > \hat{k}$ by Proposition \ref{THM_away_zero}(iii). 
	Therefore, $\epsilon_i^{k}, i\in\mathcal{Z}(X^{k})$ is never reduced after finite iterations.  This contradiction completes the
	proof.
\end{proof}


\subsection{Rank Identification}

The following theorem establishes the rank identification property of \Cref{alg_acc}, which is a straightforward result from Proposition \ref{THM_away_zero} and \Cref{thm.update}. It asserts the rank of the iterates generated by \Cref{alg_acc} will eventually remain fixed, and is equivalent to the rank of the cluster point $X^*$.  Moreover, all cluster points have the same rank. Consequently, the iterates $\{X^{k}\}$ will eventually reside in a low-dimensional active manifold $\mathcal{M}(X^*):= \{X \in \mathbb{R}^{m\times n}\mid \Rank(X) = \Rank (X^{*})\}$, implying the original problem eventually reverts to a smooth problem after identifying an active manifold $\mathcal{M}(X^{*})$.

\begin{theorem}\label{Theo:RankIdentification}
	Suppose Assumptions \ref{ASM_fLip}-\ref{ASM_levelBoundness} hold true.  Let $\{X^{k}\} $ be the sequence generated by \Cref{alg_acc}. Then $\Rank(X^{k}) = r^*:=|\Ical^*|$ for sufficiently large $k$.  Moreover, for any limit point $X^*$ of $\{X^k\}$,   $\Rank(X^*) =|\Ical(X^*)| = r^*$,  and  $\sigma_i(X^{*})\ge\left(\frac{\lambda p}{2\beta C_1}\right)^{\frac{1}{1-p}} > 0$, $i\in\Ical(X^*)$.   
\end{theorem}

 \begin{remark}
	\Cref{Theo:RankIdentification} suggests that the proposed \Cref{alg_acc} identifies the rank of the optimal solution after finite iterations.  That is, all subsequent iterates $X^{k}$ satisfy $X^{k} \in \mathcal{M}(X^*)$. Moreover,  all limit points of iterates will be confined to a low-dimensional manifold. This feature of the proposed algorithm represents a stark contrast to the results established in \cite[Theorem 6]{lee2023accelerating}, where the rank identification property merely holds for any convergent subsequence. 
\end{remark}

We illustrate the rank identification property of \Cref{alg_acc} through a simple example. 

\begin{example}
	Consider
	\begin{equation}\label{ex_toy}
		\min_{X \in \mathbb{R}^{15\times 15}} \, \frac{1}{2}\|\mathbf{P}_{\Omega}(X)-\mathbf{P}_{\Omega}(X^*)\| + \lambda \|X\|_{p}^{p},
	\end{equation}
	where $X^{*}\in\mathbb{R}^{15\times 15}$ with $\Rank(X^{*})=3$ is the ground-truth matrix to be found, $\Omega$ denotes the random sample set with sampling ratio $\text{SR} =0.5$ and $|\Omega|$ satisfies $ |\Omega| = \lceil 15^2 * \text{SR}\rceil$, $\mathbf{P}_{\Omega}$ is the projection onto the subspace of sparse matrices with nonzeros restricted to the index set $\Omega$, and $\lambda = 0.1$, $p=0.5$.
\end{example}

\begin{figure}[htbp]
	\centering{
		\includegraphics[scale=0.4]{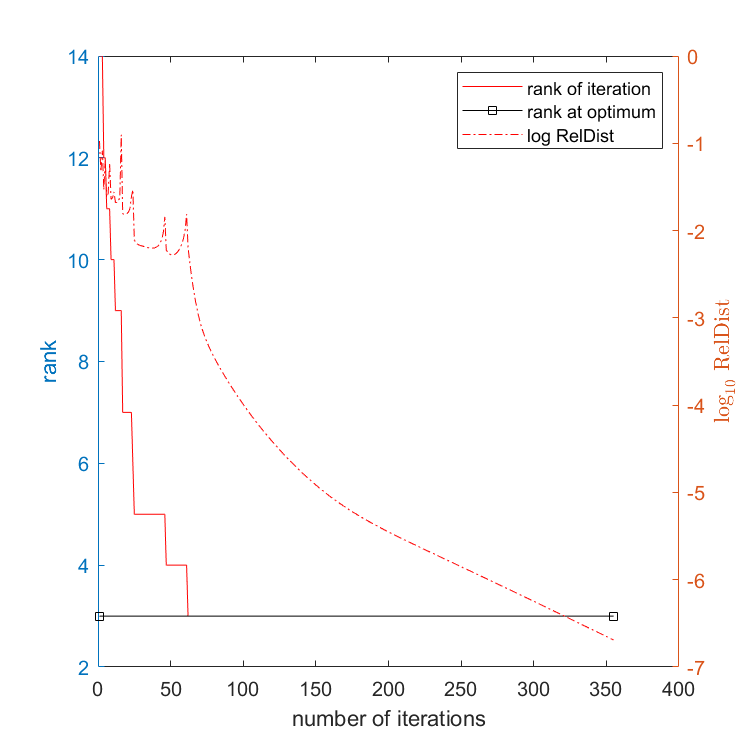}
		\caption{A sample example to show the rank-identification property.}
		\label{fig_rank_identification}}
\end{figure}

The $y$-axis on the left represents the rank and the one on the right represents 
the relative residual.  We use solid lines and dash lines to show the rank of the iterates and relative residual $\|X^k-X^*\|_F^2/\|X^0-X^*\|_F^2$ of our method, respectively.  The gray line represents the rank at the optimum $X^{*}$. We can see from \Cref{fig_rank_identification} when the rank reaches $\rm{rank}(X^{*})$, the relative distance also drops significantly, indicating that it finds a solution with high precision on the smooth manifold. 


\subsection{Global Convergence}
By \Cref{Prop_Psubdifferential}, the necessary optimality condition of \eqref{pro_pri} is given by:
\begin{equation}\label{KKT_condition}
	\nabla{f}(X^{*}) + \lambda U^{*}{\rm diag}(\bm{w}^{*}\circ \bm{\xi}^{*}){V^{*}}^{\top} = \bm{0},  
\end{equation}
where $\bar{\bm{w}}^{*}\in\partial \Vert\bm{\sigma}(X^*)\Vert_{p}^{p}   $, $\bm{\xi}^{*}\in\partial |\bm{\sigma}(X^*)|$,  and $(U^{*},V^{*}) \in \overline{\mathcal{M}}(X^{*})$\\

To show the global convergence properties of \Cref{alg_acc},  we 
investigate the optimality error at $X^{k+1}$. 

\begin{theorem}[Bounded subgradients]\label{cvgc_Global_ACC_sub}
	Suppose Assumptions \ref{ASM_fLip}-\ref{ASM_levelBoundness} hold true. Let $\{X^{k}\} $ be the sequence generated by \Cref{alg_acc}. For each $k\in\mathbb{N}$, define the optimality error associated with \eqref{pro_pri} as
	\begin{equation}\label{KKT_residuals_define}
		E^{k+1} = \nabla{f}(X^{k+1}) + \lambda U^{k+1}{\rm diag}(\bar{\bm{w}}^{k+1}\circ\bm{\xi}^{k+1}) {V^{k+1}}^{\top},
	\end{equation}
	where $\bar{w}_{i}^{k+1} = p ( {\sigma}_i(X^{k+1}) )^{p-1}, i\in\mathcal{I}(X^{k+1})$, $\bar{w}_{i}^{k+1}= w_i^k, i\in\mathcal{Z}(X^{k+1})$ and $\bm{\xi}^{k+1}\in\partial\vert\sigma(X^{k+1})\vert$. Then it holds that $E^{k+1} \in \partial F(X^{k+1})$ and there exists $C_2 > 0 $ such that
	\begin{equation}
		\Vert E^{k+1} \Vert_F \leq (L_{f}+2\beta+C_2)\|X^{k+1}-X^{k}\|_{F} + (L_{f}+\beta)\bar{\alpha}\|X^{k}-X^{k-1}\|_{F} + C_2\epsilon_{1}^0,
	\end{equation}
	with $C_2 = mp(1-p)(C_3)^{p-2}$ and $0 < C_3 \leq \min_{i \in [\vert \mathcal{I}(X^{k+1})\vert]} \sigma_{i}(X^{k+1}),\forall k\in \mathbb{N}$. 
	\begin{proof}
		By \Cref{Prop_Psubdifferential}, we know that $\bar{\bm{w}}^{k+1} \in \partial\Vert\bm{\sigma} (X^{k+1})\Vert_{p}^{p}$. This, together with differentiability of $f$ \cite{rockafellar2009variational}[10.10 Exercise], leads to the desired result $E^{k+1} \in \partial F(X^{k+1})$. On the other hand, by subtracting \eqref{KKT_acc_opt} from both sides of \eqref{KKT_residuals_define}, we have
		\begin{equation}\label{KKT_residuals}
			\begin{aligned}
				E^{k+1} 
				=&\ [\nabla{f}(X^{k+1})  - \nabla{f}(Y^{k})] - \beta [(X^{k+1}-Y^{k}) + (X^{k+1} - X^{k})] \\ 
				&\ + \lambda U^{k+1}{\rm diag}((\bar{\bm{w}}^{k+1} - \bm{w}^{k}) \circ\bm{\xi}^{k+1}) {V^{k+1}}^{\top}.  
			\end{aligned}
		\end{equation}
		It follows from \Cref{ASM_fLip} and \eqref{eq: extrapolatedPoint} that the first term in \eqref{KKT_residuals} 
		\begin{equation} \label{ineq_acc_nabdaF}
			\begin{aligned}
				\|\nabla f(X^{k+1}) - \nabla f (Y^{k})\|_{F} & \le L_{f}\|X^{k+1}-Y^{k}\|_{F} \\ 
				&\overset{(a)}{\leq} L_{f} \|X^{k+1}-X^{k}\|_{F} + L_{f}\bar{\alpha}\|X^{k}-X^{k-1}\|_{F},
			\end{aligned}
		\end{equation}
		where $(a)$ holds due to the triangle inequality and \eqref{alpha_update}. Similarly, we have from the second term in \eqref{KKT_residuals} that
		\begin{equation} \label{ineq_acc_nabdaF1}
			\begin{aligned}
				\Vert \beta [(X^{k+1}-Y^{k}) + (X^{k+1} - X^{k})] \Vert_{F} &\leq 2\beta \Vert X^{k+1} - X^{k} \Vert_{F} + \beta\bar{\alpha}\Vert X^k - X^{k-1}\Vert_{F}.
			\end{aligned}
		\end{equation}
		As for the third term in \eqref{KKT_residuals}, we have
		\begin{equation}\label{eq.usv}
			\begin{aligned} 
				&\quad \ \| U^{k+1}{\rm diag}((\bar{\bm w}^{k+1} - \bm{w}^{k}) \circ\bm{\xi}^{k+1}) {V^{k+1}}^{\top}\|_F \\
				&=  \| {\rm diag}((\bar{\bm w}^{k+1} - \bm{w}^{k}) \circ\bm{\xi}^{k+1})  \|_F \\
				&\overset{(a)}{\leq}  \|  \bar{\bm w}^{k+1} - \bm{w}^{k}\|_2 \leq \|  \bar{\bm w}^{k+1} - \bm{w}^{k}\|_1\\ 
				&\overset{(b)}{=} \sum\limits_{i=1}^{\vert \mathcal{I}(X^{k+1})\vert}p (1-p) \left( \hat{\sigma}_i(X^k) \right)^{p-2}  \left(\sigma_i(X^{k+1})- (\sigma_i(X^{k}) + \epsilon_{i}^{k}) \right) \\
				&\overset{(c)}{\leq} \sum\limits_{i=1}^{\vert \mathcal{I}(X^{k+1})\vert}p (1-p) \left(\hat{\sigma}_i(X^k)\right)^{p-2}  \left(\left|\sigma_i(X^k)-\sigma_i(X^{k+1}) \right| + \vert \epsilon_{i}^{k}\vert \right)\\
				&\overset{(d)}{\leq}  \sum\limits_{i=1}^{\vert \mathcal{I}(X^{k+1})\vert}p(1-p) \left(\hat{\sigma}_{\vert \mathcal{I}(X^{k+1})\vert}(X^{k})\right)^{p-2}  \left(\| X^k-X^{k+1} \|_F + \epsilon_{i}^{k} \right), 
			\end{aligned} 
		\end{equation} 
		where $(a)$ holds due to $\xi_{i}^{k+1} \in [-1,1], \forall i\in[m]$, equality $(b)$ makes use of the mean value theorem with $\hat{\sigma}_i(X^{k}) $ lying between $\sigma_i(X^k)+ \epsilon_{i}^k$ with $\sigma_i(X^{k+1})$ for each $i\in [m]$, inequality $(c)$ is true by triangle inequality and inequality $(d)$ is true because of the monotonicity of $(\cdot)^{p-2}$ over $\mathbb{R}_{++}$ and conclusions drawn by \cite[Problem 7.3]{horn2012matrix}. 
		
		Therefore, combining \eqref{ineq_acc_nabdaF}, \eqref{ineq_acc_nabdaF1} with \eqref{eq.usv} leads to 
		\begin{equation}
			\begin{aligned}
				\Vert E^{k+1} \Vert_{F} &\leq \|\nabla f(X^{k+1}) - \nabla f (Y^{k})\|_{F}+ \Vert \beta [(X^{k+1}-Y^{k}) + (X^{k+1} - X^{k})] \Vert_{F}\\
				&\quad +\| U^{k+1}{\rm diag}((\bar{\bm w}^{k+1} - \bm{w}^{k}) \circ\bm{\xi}^{k+1}) {V^{k+1}}^{\top}\|_F\\
				&\overset{(a)}{\leq} (L_{f}+2\beta)\|X^{k+1}-X^{k}\|_{F} + (L_{f}+\beta)\bar{\alpha}\|X^{k}-X^{k-1}\|_{F}\\
				&\quad + \sum\limits_{i=1}^{\vert \mathcal{I}(X^{k+1})\vert}p(1-p) \left( C_3\right)^{p-2}  \left(\| X^k-X^{k+1} \|_F + \epsilon_{i}^{k} \right)\\
				&\overset{(b)}{\leq} (L_{f}+2\beta+mp(1-p)(C_3)^{p-2})\|X^{k+1}-X^{k}\|_{F} + (L_{f}+\beta)\bar{\alpha}\|X^{k}-X^{k-1}\|_{F}\\
				&\quad + (mp(1-p)(C_3)^{p-2})\epsilon_{1}^0,
			\end{aligned}
		\end{equation}
		where $(a)$ holds since $0 < C_3 \leq \min_{i \in [\vert \mathcal{I}(X^{k+1})\vert]} \sigma_{i}(X^{k+1}),\forall k\in \mathbb{N}$ and $(b)$ is true because $\vert \mathcal{I}(X^k)\vert \leq m, \forall k \in \mathbb{N}$ and $\bm{\epsilon}^{0} \in \mathbb{R}^{m}_{\downarrow}\cap\mathbb{R}^{m}_{++}$. This completes the proof.
	\end{proof}
\end{theorem}

Based on the previous analysis, we now demonstrate the global convergence properties of the EIRNRI algorithm. Before proceeding, we use  $\chi^{\infty} $ to denote the cluster point of $\{X^{k}\}$ generated by \Cref{alg_acc}.

\begin{theorem}\label{Global_sub}
	Suppose Assumptions \ref{ASM_fLip}-\ref{ASM_levelBoundness} hold true. Let $\{X^{k}\} $ be the sequence generated by \Cref{alg_acc}. Then the following assertions hold.
	\begin{itemize}\label{Theorem:Global}
		\item[(i)] The set of cluster points $\chi^{\infty}$  is nonempty, compact and connected set. 
		\item[(ii)] $\{F(X^k)\}$ is convergent. Moreover, the objective function $F$ is constant on $\chi^{\infty}$.
		\item[(iii)] $\lim\limits_{k\to+\infty}\|E^{k+1} \|_F =  0$. Therefore, any cluster point $X^{*}\in\chi^{\infty}$ is a critical point of $F$, meaning $\chi^{\infty} \subset \textrm{crit}(F)$.
	\end{itemize}
\end{theorem}

\begin{proof}
	(i) Lemma \ref{lem_Dec_acc}(iii) implies $\chi^{\infty}$ is nonempty. On the other hand, Lemma \ref{lem_Dec_acc}(ii)-(iii), combined with the classical Ostrowski result \cite[Theorem 26.1]{ostrowski1973solution}, leads to the desired results. This proves statement (i).
	
	(ii) The convergence of $\{F(X^k)\}$ is direct consequence of \eqref{pro_prox} by invoking \Cref{thm.update}(ii) and Lemma \ref{lem_Dec_acc}(ii). On the other hand, it holds from \Cref{Theo:RankIdentification} and Lemma \ref{lem_Dec_acc}(i)-(iii) 
	that for each $X^{*} \in \chi^{\infty}$ with $\Rank(X^*)=r^{*}$,
	\begin{equation*}
		\begin{aligned}
			F(X^{*}) 
			=&\ f (X^{*}) + \lambda\sum\limits_{i=1}^{r^{*}}(\sigma_i(X^*))^{p} \\
			\overset{(a)}{=}&\ \lim\limits_{k\to+\infty } \left[ f(X^{k+1}) + \lambda\sum\limits_{i=1}^{m}(\sigma_i(X^{k+1}) +\epsilon_{i}^{k+1})^{p} + \beta\|X^{k+1}-X^{k}\|_{F}^{2}\right] - \lambda \sum\limits_{i=r^{*} +1}^{m} (\epsilon_{i}^{\hat{k}})^{p}  \\
			=&\ \lim\limits_{k\to+\infty  } H(X^{k+1},X^{k},\epsilon^{k+1})  - \lambda \sum\limits_{i=r^{*} +1}^{m} (\epsilon_{i}^{\hat{k}})^{p} \\
			\overset{(b)}{=}&\ H^*  - \lambda \sum\limits_{i=r^{*} +1}^{m} (\epsilon_{i}^{\hat{k}})^{p},
		\end{aligned}
	\end{equation*} 
	where $(a)$ is true because \Cref{thm.update} guarantees that there exists $\hat{k} \in \mathbb{N}$ such that $\bm{\epsilon}^k\to \bm{\epsilon}^* = [0,\cdots,0,\epsilon_{r^{*}+1}^{\hat{k}},\cdots,\epsilon_{m}^{\hat{k}}]^{\top}$ as $k\to+\infty$, and also by Lemma \ref{lem_Dec_acc}(ii), and $(c)$ holds simply by Lemma \ref{lem_Dec_acc}(i). 
	
	(iii) 	Recall \eqref{ineq_acc_nabdaF} and \eqref{ineq_acc_nabdaF1}, we can deduce from Lemma \ref{lem_Dec_acc}(iii) that 
	\begin{equation}\label{eq_1st_term_vanish}
		\lim\limits_{k\to+\infty}\|\nabla f(X^{k+1}) - \nabla f (Y^{k})\|_{F}=0
	\end{equation}
	and
	\begin{equation}\label{eq_2nd_term_vanish}
		\lim\limits_{k\to+\infty} \Vert \beta [(X^{k+1}-Y^{k}) + (X^{k+1} - X^{k})] \Vert_{F} = 0.
	\end{equation}
	On the other hand, by Proposition \ref{THM_away_zero}(iii) and monotonicity of $(\cdot)^{p-2}$ over $\mathbb{R}_{++}$, we have from \eqref{eq.usv} that for sufficiently large $k \in \mathbb{N}$
	\begin{equation}\label{eq_3rd_term_vanish0}
		\begin{aligned}
			&\quad \ \| U^{k+1}{\rm diag}((\bar{\bm w}^{k+1} - \bm{w}^{k}) \circ\bm{\xi}^{k+1}) {V^{k+1}}^{\top}\|_F^2 \\
			&\leq \sum\limits_{i=1}^{\vert \mathcal{I}(X^{k+1})\vert}p(1-p) \left(\hat{\sigma}_{\vert \mathcal{I}(X^{k+1})\vert}(X^{k})\right)^{p-2}  \left(\| X^k-X^{k+1} \|_F + \epsilon_{i}^{k} \right)\\
			& <  \sum\limits_{i=1}^{\vert \mathcal{I}(X^{k+1})\vert}p(1-p) \left( \left(\frac{\lambda p}{2\beta C_1}\right)^{1/(1-p)}-\epsilon_{i}^k\right)^{p-2}  \left(\| X^k-X^{k+1} \|_F + \epsilon_{i}^{k} \right).
		\end{aligned}
	\end{equation}
	By \Cref{thm.update}(ii) and Lemma \ref{lem_Dec_acc}(iii), we know from \eqref{eq_3rd_term_vanish0} that
	\begin{equation}\label{eq_3rd_term_vanish1}
		\lim\limits_{k\to+\infty} \| U^{k+1}{\rm diag}((\bar{\bm w}^{k+1} - \bm{w}^{k}) \circ\bm{\xi}^{k+1}) {V^{k+1}}^{\top}\|_F^2 = 0.
	\end{equation}
	Combining \eqref{eq_1st_term_vanish}, \eqref{eq_2nd_term_vanish} and \eqref{eq_3rd_term_vanish1} yields
	\begin{equation*}
		\lim\limits_{k\to+\infty}\|E^{k+1} \|_F =  0,
	\end{equation*}
	meaning $\bm{0} \in \partial F(X^{*})$ for any $X^{*} \in \chi^{\infty}$, as desired. This completes the proof.
\end{proof}

\section{Convergence Analysis Under KL Property}

In this section, we analyze the convergence properties of the sequence $\{ (X^{k}, Y^k, \bm{\epsilon}^{k})\} $ generated by \Cref{alg_acc} under the KL property of $F$ for sufficiently large $k$. We first define a reduced form of \eqref{pro_prox} as 
\begin{equation}\label{pro_reduce_acc}
	\hat{H} (X,Y,\bm{\delta}) = f (X) +\sum\limits_{i=1}^{r^{*}} \left(\sigma_{i} (X)+ (\delta_{i})^{2}\right)^{p} + \frac{\beta}{2}\|X-Y\|_{F}^{2}, 
\end{equation}
where   $\epsilon_{i} = \delta_{i}^{2}$ with $\delta_{i}\ge0$ since $\epsilon_{i} $ is restricted to be non-negative. 
Notice that by \Cref{thm.update}, $\delta_i^k \equiv \delta_i^{\hat k}, i\in\mathcal{Z}^*$ and $\delta_i^{k+1} = \sqrt{\mu}\delta_i^k, i\in \Ical^*$ for all sufficiently large $k$.  
We consider the Cartesian product of triplets $(X,Y,\bm{z}) \in \mathbb{R}^{m\times n} \times \mathbb{R}^{m\times n}\times \mathbb{R}^{r^{*}}$.
\begin{definition}\label{def_product} 
	Consider a set $\mathbb{S} = \{(X,Y,\bm{z}) \mid (X,Y,\bm{z}) \in \mathbb{R}^{m\times n} \times \mathbb{R}^{m\times n}\times \mathbb{R}^{r^{*}}\}$. Define the Cartesian product of any $\bm{X} = (X_1, X_2, \bm{x_3}) \in \mathbb{S}$ and $\bm{Y} = (Y_1, Y_2, \bm{y_3}) \in \mathbb{S}$ as 
	\begin{equation*}
		\bm{X} \times \bm{Y} = \langle X_{1},Y_{1} \rangle + \langle X_{2},Y_{2} \rangle + \langle \bm{x_{3}}, \bm{y_{3}} \rangle.
	\end{equation*}
	The norm of any $\bm{X}\in\mathbb{S} $ is defined as 
	\begin{equation*} 
		{\left\|X\right\|}  =  ({\|X_{1}\|}^{2}_{F} +{\|X_{2}\|}^{2}_{F}+ {\|\bm{x_{3}}\|}^{2}_{2})^{1/2}, 
	\end{equation*}
	and thus the distance between $\bm{X}\in\mathbb{S}$ and $\bm{Y}\in\mathbb{S}$ is 
	\begin{equation*} 
		\mathrm{dist} \left(X,Y\right)  = \sqrt{{\|X_{1}-Y_{1}\|}^{2}_{F} +{\|X_{2}-Y_{2}\|}^{2}_{F}+ {\|\bm{x_{3}}-\bm{y_{3}}\|}^{2}_{2} }.
	\end{equation*}
\end{definition}

We assume that the uniform KL property holds for $\hat H$ in this Cartesian product space. 

\begin{assumption}\label{ASM_KL_acc}
	Suppose  $\hat{H}:\mathbb{S}\to \mathbb{R} $ satisfies the uniform KL property on
	$\Omega := \{(X^{*},X^{*},\bm{0}_{r^{*}})\mid X^{*}\in\mathrm{crit} (F)\}$.
\end{assumption}
\begin{remark}
	The assumption that $\hat{H}$ satisfies uniform KL property at any point of $\Omega$ is indeed no stronger than the assumption that it satisfies the KL property.  Note that the compactness of $\chi^{\infty}$ by \Cref{Theorem:Global}(i) implies that $\Omega$ is compact, and additionally, \Cref{thm.update} and \Cref{Theorem:Global}(ii) indicates that $\hat{H}$ is constant on $\Omega$. Furthermore, the KL property is generally satisfied across a broad range of function classes \cite{attouch2010proximal}. Therefore, \Cref{ASM_KL_acc} can be considered reasonable and mild.
\end{remark}

Now we prove the convergence properties of $\{(X^k, Y^k, \bm{\delta}^k)\}$ using KL property.

\begin{lemma}[Uniqueness of convergence properties under KL condition] \label{lem_KLacc_cvgc_cond} 
	Suppose Assumptions \ref{ASM_fLip}-\ref{ASM_KL_acc} hold true. 
	Let $\{X^{k}\}$ be the sequence generated by \Cref{alg_acc}. There then exists $\hat{k} \in \mathbb{N}$ such that the following statements hold.
	\begin{itemize}
		\item[(i)] There exists $D>0 $ such that for all $k\ge\hat{k} $
		\begin{equation*}
			\left\| \nabla \hat{H} (X^{k+1},X^{k},\bm{\delta}^{k+1} )\right\|_{F}  
			\le D  \left(\left\|X^{k}-X^{k+1}\right\|_{F}+\left\|X^{k-1}-X^{k}\right\|_{F} + \|\bm{\delta}^{k}\|_{1}-\|\bm{\delta}^{k+1}\|_{1} \right). 
		\end{equation*}
		Moreover, $\lim\limits_{k\to+\infty} \Vert \nabla \hat{H} (X^{k+1},X^{k},\bm{\delta}^{k+1})\Vert_F = 0$.
		\item[(ii)] $\{\hat{H} (X^{k},Y^{k-1},\bm{\delta}^{k})\} $ is monotonically decreasing and there exists $D_{1} $ such that 
		\begin{equation*}
			\hat{H} (X^{k},Y^{k-1},\bm{\delta}^{k} ) - \hat{H} (X^{k+1},Y^{k},\bm{\delta}^{k+1} ) \ge D_{1} \|X^{k}-X^{k-1}\|_{F}^{2}.
		\end{equation*}
		\item[(iii)] $\hat{H} (X^{*},X^{*},\bm{0}) =\zeta :=\lim\limits_{k\to+\infty} \hat{H} (X^{k},Y^{k-1},\bm{\delta}^{k} ) $,  where $ (X^{*},X^{*},\bm{0})\in\Gamma $ with $\Gamma $ being the set of cluster points of $\{ (X^{k},Y^{k-1},\bm{\delta}^{k})\}$, that is, $\Gamma:=\left\{ (X^{*},X^{*},\bm{0}): X^{*}\in\chi^{\infty} \right\} $.  
		\item[(iv)] For any $t$,  $T^t := \sum\limits_{k=t}^{+\infty}\|X^{k-1}-X^{k}\|_{F} < +\infty $.  Therefore,  $\lim\limits_{k\to+\infty} X^k = X^*$. 
	\end{itemize}
\end{lemma}
\begin{proof}
	Since  $\hat H$ is differentiable with respect to its all input variables $(X,Y,\bm{\delta})$ separately, we have
	\begin{equation}\label{eq:derivative_H}
		\begin{aligned}
			\nabla_{X} \hat{H} (X,Y,\bm{\delta}) &= \nabla f (X)+\beta (X-Y)+\lambda U\textrm{diag}(\hat{\bm w})V^{\top},\\
			\nabla_{Y} \hat{H} (X,Y,\bm{\delta}) &= \beta (X-Y),\\
			\nabla_{\delta_{i}}\hat{H} (X,Y,\bm{\delta}) &= 2\lambda p\delta_{i} (\sigma_{i} (X)+\delta_{i}^{2})^{p-1}, \ \forall i \in [r^{*}],
		\end{aligned}
	\end{equation}
	where $\hat{\bm w} =(p (\sigma_{1} (X)+(\delta_{1})^{2})^{p-1},\cdots,p (\sigma_{r^{*}} (X)+(\delta_{r^{*}})^{2})^{p-1},0,\cdots,0)^{\top} \in \mathbb{R}^{m}$ by \Cref{Prop_Psubdifferential}. 
	
	For each $k > \hat{k}$,  we have from \eqref{KKT_acc_opt} and Proposition \ref{THM_away_zero}(ii)-(iii) that
	\begin{equation}\label{KKT_acc_reform}
		\bm{0}= \nabla f (Y^{k}) + \beta (X^{k+1}-Y^{k}) + \beta(X^{k+1} - X^{k}) + \lambda U^{k+1}\mathrm{diag} \left(\hat{\bm{w}}^{k}\right){V^{k+1}}^{\top},
	\end{equation}
	where $\hat{\bm{w}}^{k} = (p(\sigma_1(X^{k}) + (\delta_1^{k})^{2})^{p-1},\cdots,p(\sigma_{r^{*}}(X^{k}) + (\delta_{r^{*}}^{k})^{2})^{p-1},0,\cdots,0)^{\top}$.
	
	(i) Combining the first expression in \eqref{eq:derivative_H} and \eqref{KKT_acc_reform}, we have
	\begin{equation}\label{partial_hatH}
		\begin{aligned} 
			& \nabla_{X} \hat{H} (X^{k+1},X^{k},\bm{\delta}^{k+1} ) \\
			=& \nabla f (X^{k+1}) - \nabla f(Y^{k}) - \beta (X^{k+1}-Y^{k}) +\lambda U^{k+1} \textrm{diag}\left(\hat{\bm{w}}^{k+1} - \hat{\bm{w}}^{k}\right){V^{k+1}}^{\top}.     
	\end{aligned}\end{equation}
	It follows from \eqref{eq: extrapolatedPoint} that
	\begin{equation}\label{ineq_acc_dxy_APD_APD}
		\|\beta (X^{k}-Y^{k})\|_{F} \le \beta\bar{\alpha}\|X^{k}-X^{k-1}\|_{F}.
	\end{equation}
	For any $k \geq \hat{k}$, we then have
	\begin{equation}\label{dec_fro_err_Left}
		\begin{aligned} 
			&\ \| U^{k+1}{\rm diag}(\hat{\bm{w}}^{k+1} - \hat{\bm{w}}^{k}) {V^{k+1}}^{\top}\|_F\\
			\leq & \  \| {\rm diag}(\hat{\bm{w}}^{k+1} - \hat{\bm{w}}^{k})\|_F = \| \hat{\bm{w}}^{k+1} - \hat{\bm{w}}^{k}\|_2  \leq \| \hat{\bm{w}}^{k+1} - \hat{\bm{w}}^{k} \|_1 \\
			\leq & \sum\limits_{i=1}^{r^{*}}  p (1-p) (\sigma_{i}(X^k))^{p-2}(\vert \sigma_{i}(X^{k})  - \sigma_{i}(X^{k+1})\vert + \vert(\delta_{i}^{k})^2 -  (\delta_{i}^{k+1})^2\vert)\\
			\overset{(a)}{\leq} & \sum\limits_{i=1}^{r^{*}}  p (1-p) \left(\frac{2\beta C_1}{\lambda p}\right)^{\frac{2-p}{1-p}}  \left(\| X^k-X^{k+1} \|_F + \vert (\delta_{i}^{k})^2 - (\delta_i^{k+1})^2\vert \right),\\ 
			\le & D_p  \left(   \| X^k-X^{k+1} \|_F   + \max\limits_{i}  (\delta_{i}^{k}+\delta_{i}^{k+1}) \|\bm{\delta}^{k}-\bm{\delta}^{k+1}\|_{1}\right) \\
			\le &  D_p \left[  \|X^{k}-X^{k+1}\|_{F} + 2\|\bm{\delta}^{0}\|_{\infty}  (\|\bm{\delta}^{k}\|_{1} - \|\bm{\delta}^{k+1}\|_{1}) \right] ,  
	\end{aligned} \end{equation}
	where $D_p = p (1-p) \left(\frac{2\beta C_1}{\lambda p}\right)^{\frac{2-p}{1-p}} $, and inequality $(a)$ holds by Proposition \ref{THM_away_zero}(iii) and conclusions drawn by \cite[Problem 7.3]{horn2012matrix}. This, together with \eqref{partial_hatH}, \eqref{ineq_acc_nabdaF}, \eqref{ineq_acc_dxy_APD_APD} and \eqref{dec_fro_err_Left}, yields
	\begin{equation}\label{eq:first_term}
		\begin{aligned} 
			&\  \| \ \nabla_{X} \hat{H} (X^{k+1},X^{k},\bm{\delta}^{k+1} )\|_F \\
			\le&\ \|\nabla f (X^{k+1})-f (Y^{k})\|_{F} + \beta\|X^{k}-Y^{k}\|_{F} + \lambda \| U^{k+1}{\rm diag}(\hat{\bm{w}}^{k+1} - \hat{\bm{w}}^{k}) {V^{k+1}}^{\top}\|_F \\ 
			\le&\  (\lambda{D_p} +L_{f})\|X^{k+1}-X^{k}\|_{F} +  (L_{f}+\beta)\bar{\alpha}\|X^{k}-X^{k-1}\|_{F} \\
			&\ + 2\lambda{D_p}\|\bm{\delta}^{0}\|_{\infty} (\|\bm{\delta}^{k}\|_{1}-\|\bm{\delta}^{k+1}\|_{1}).      
		\end{aligned}
	\end{equation}
	Similarly,  we have
	\begin{equation}\label{eq:second_term}
		\| \nabla_{Y} \hat{H} (X^{k+1},X^{k},\bm{\delta}^{k+1} ) \|_{F} = \|\beta(X^{k+1} - X^{k})\|_F \le \beta\|X^{k+1}-X^{k}\|_{F},
	\end{equation} 
	and  note that
	$$\nabla_{\bm{\delta}}\hat{H} (X^{k+1},X^{k},\bm{\delta}^{k+1}) = 2\lambda \hat{\bm w}^{k+1}\circ\bm{\delta}^{k+1}.$$ 
	It then follows that 
	\begin{equation}\label{eq:third_term}
		\begin{aligned} 
			\ \| \nabla_{\bm{\delta}}\hat{H} (X^{k+1},X^{k},\bm{\delta}^{k+1})\|_{2} &\le \| \nabla_{\bm{\delta}}\hat{H} (X^{k+1},X^{k},\bm{\delta}^{k+1})\|_{1} \\
			&= \sum\limits_{i=1}^{r^{*}} 2\lambda  \hat{w}_{i}^{k+1} {\delta}_{i}^{k+1} \\
			&\overset{(a)}{\leq} \sum\limits_{i=1}^{r^{*}} 2\lambda (\tfrac{2\beta C_1}{\lambda}) \tfrac{\sqrt{\mu}}{1-\sqrt{\mu}} \left(\delta_{i}^{k}-\delta_{i}^{k+1}\right) \\
			&\le \tfrac{4\beta C_1\sqrt{\mu}}{1-\sqrt{\mu}} \left(\|\bm{\delta}^{k}\|_{1} - \|\bm{\delta}^{k+1}\|_{1} \right),
		\end{aligned} 
	\end{equation}
	where inequality $(a)$ holds by Proposition \ref{THM_away_zero}(ii) and  $\delta_{i}^{k+1}\le \sqrt{\mu}\delta_{i}^{k}, i\in[r^{*}] $.
	
	Therefore, combining \eqref{eq:first_term}, \eqref{eq:second_term} and \eqref{eq:third_term} gives 
	\begin{equation*}\label{KLacc_K1}
		\| \nabla \hat{H} (X^{k+1},X^{k},\bm{\delta}^{k+1})\|_{F} \le D  \left(\left\|X^{k}-X^{k+1}\right\|_{F}+\left\|X^{k-1}-X^{k}\right\|_{F} + \|\bm{\delta}^{k}\|_{1}-\|\bm{\delta}^{k+1}\|_{1} \right)
	\end{equation*}
	with $ D = \max \left(\lambda{D_p} +L_{f}+ \beta, (L_f+\beta)\bar{\alpha},2D_p\lambda\|\bm{\delta}^{0}\|_{\infty}+\frac{4\beta C\sqrt{\mu}}{1-\sqrt{\mu}} \right) < +\infty$. Furthermore, by \Cref{lem_Dec_acc} and \Cref{thm.update}, we know that $\lim\limits_{k\to+\infty} \Vert\nabla\hat{H} (X^{k+1},X^{k},\bm{\delta}^{k+1} )\Vert_F = 0$. This completes the proof of statement (i).

	(ii)  It is straightforward   from $D_{1}=\frac{\beta}{2} \left(1-\frac{3L_{f}+\beta}{\beta}\bar{\alpha}^{2}\right)>0$ by Lemma \ref{lem_Dec_acc}(i). 
	
	(iii) \Cref{Global_sub}(ii) shows $H(X^{k+1},X^{k},\bm{\delta}^{k+1} )\to\zeta$ as $k\to+\infty$ and It follows from Lemma \ref{lem_Dec_acc}(iii) and \Cref{Global_sub}(i)-(ii) that $\{\hat{H} (X^{k+1},Y^{k},\bm{\delta}^{k+1})\}$ uniquely converges. 
	
	(iv) By statement (iii), we know that
	\begin{equation*}
		\lim\limits_{k \to +\infty} \hat{H}(X^{k+1},X^{k},\bm{\delta}^{k+1}) = \hat{H}(X^{*},X^{*},\bm{0}) \equiv \zeta.
	\end{equation*}
	If $\hat{H} (X^{k+1},X^{k},\bm{\delta}^{k+1})=\zeta$ for each $k > \hat{k}$, then we know $X^{k+1} = X^{k} $ by statement (ii) after $\hat{k}$, indicating $X^{k} = X^{\hat{k}} \in \chi^{\infty}$. The proof then is complete. We otherwise have to consider the case in which $\hat{H} (X^{k+1},X^{k},\bm{\delta}^{k+1})>\zeta $ after $\hat{k}$.
	
	By \Cref{Uniform_KL} and \Cref{ASM_KL_acc}, we know that there exist a desingularizing function $\Phi$, $\eta>0$ and $\rho>0$ such that  
	\begin{equation}\label{ineq_KLacc_K1}
		\Phi{'} \left(\hat{H} (X,Y,\bm{\delta})-\zeta\right)  \|\nabla \hat{H} (X,Y,\bm{\delta})\| \ge 1,
	\end{equation}
	for all $ (X,Y,\bm{\delta})\in\mathbb{U} ( (X^{*},X^{*},\bm{0});\rho)\cap\left\{ (X,Y,\bm{\delta})\in\mathbb{S}:\zeta<\hat{H} (X^{k+1},X^{k},\bm{\delta}^{k+1})<\zeta+\eta\right\} $. Note that $X^{*} \in \chi^{\infty}$, we have
	\begin{equation*}
		\lim\limits_{k\to+\infty} \mathrm{dist} ( (X^{k},X^{k-1},\bm{\delta}^{k}),\Gamma) = 0,
	\end{equation*}
	meaning $\forall \rho > 0$, there exists $k_{1}\in\mathbb{N} $ such that  $\mathrm{dist} ( (X^{k},X^{k-1},\bm{\delta}^{k}),\Gamma) < \rho $ for all $k>k_{1} $.   
	Since $\{ \hat{H} (X^{k+1},X^{k},\bm{\delta}^{k+1})\} \to \zeta$ as $k \to +\infty$, we know that there exists $k_{2}\in\mathbb{N} $ such that $\zeta<\hat{H} (X^{k+1},X^{k},\bm{\delta}^{k+1})<\zeta+\eta $ for all $k>k_{2} $. Thus, we have from the smoothness of $\hat{H}$ that for any  $k>\max \left(k_{1},k_{2}\right)$, 
	\begin{equation}
		\Phi{'} \left(\hat{H} (X^{k+1},X^{k},\bm{\delta}^{k+1})-\zeta\right) \| \nabla \hat{H} (X^{k+1},X^{k},\bm{\delta}^{k+1})\| \ge 1.
	\end{equation} 
	Then, we have for any $k > \hat{k}$ that 
	\begin{equation}\label{eq:onebound}
		\begin{aligned}
			&  \left[\Phi \left(\hat{H} (X^{k},X^{k-1},\bm{\delta}^{k})-\zeta\right) - \Phi \left(\hat{H} (X^{k+1},X^{k},\bm{\delta}^{k+1})-\zeta\right)\right] \\
			&\times  D \left(\|X^{k-2}-X^{k-1}\|_{F} +\|X^{k-1}-X^{k}\|_{F} + \|\bm{\delta}^{k-1}\|_{1}-\|\bm{\delta}^{k}\|_{1}\right) \\
			\overset{(a)}{\geq}& \left[\Phi \left(\hat{H} (X^{k},X^{k-1},\bm{\delta}^{k})-\zeta\right) - \Phi \left(\hat{H} (X^{k+1},X^{k},\bm{\delta}^{k+1})-\zeta\right)\right] \\
			&\times  \|\nabla \hat{H} (X^{k},X^{k-1},\bm{\delta}^{k})\| \\
			\overset{(b)}{\geq}&\left[\hat{H} (X^{k},X^{k-1},\bm{\delta}^{k})-\hat{H} (X^{k+1},X^{k},\bm{\delta}^{k+1})\right] \Phi{'} \left(\hat{H} (X^{k},X^{k-1},\bm{\delta}^{k})-\zeta\right) \\
			&\times  \|\nabla \hat{H} (X^{k},X^{k-1},\bm{\delta}^{k})\| \\
			\overset{(c)}{\geq}& D_{1} \|X^{k}-X^{k-1}\|_{F}^{2} ,
	\end{aligned}\end{equation}   
	where the inequality $(a)$ holds by Lemma \ref{lem_KLacc_cvgc_cond}(i), the inequality $(b)$ makes use of the concavity of $\Phi$ and the inequality $(c)$ follows from \eqref{ineq_KLacc_K1} and Lemma \ref{lem_KLacc_cvgc_cond}(ii). Therefore, 
	\begin{equation}\label{eq_anotherbound}
		\begin{aligned}
			&\ \|X^{k}-X^{k-1}\|_{F} \\
			\overset{(a)}{\leq}&\ \sqrt{\frac{2D}{D_{1}}\left[\Phi \left(\hat{H} (X^{k},X^{k-1},\bm{\delta}^{k})-\zeta\right) - \Phi \left(\hat{H} (X^{k+1},X^{k},\bm{\delta}^{k+1})-\zeta\right)\right]} \\
			&\ \times \sqrt{\frac{1}{2} \left(\left\|X^{k-2}-X^{k-1}\right\|_{F} +\left\|X^{k-1}-X^{k}\right\|_{F} + \|\bm{\delta}^{k-1}\|_{1} - \|\bm{\delta}^{k}\|_{1}\right)} \\
			\overset{(b)}{\leq}&\ \frac{D}{D_{1}}\left[\Phi \left(\hat{H} (X^{k},X^{k-1},\bm{\delta}^{k})-\zeta\right) - \Phi \left(\hat{H} (X^{k+1},X^{k},\bm{\delta}^{k+1})-\zeta\right)\right] \\
			&\ + \frac{1}{4} \left(\left\|X^{k-2}-X^{k-1}\right\|_{F} +\left\|X^{k-1}-X^{k}\right\|_{F} + \|\bm{\delta}^{k-1}\|_{1}-\|\bm{\delta}^{k}\|_{1}\right),
		\end{aligned}
	\end{equation}
	where the inequality $(a)$ holds by \eqref{eq:onebound} and the inequality $(b)$ is true because of the AM-GM inequality. Then, subtracting $\frac{1}{2}\|X^{k}-X^{k-1}\|_{F} $ from both sides of  \eqref{eq_anotherbound}, we obtain 
	\begin{equation}\label{eq:thirdbound}
		\begin{aligned}
			\frac{1}{2}\|X^{k}-X^{k-1}\|_{F} \le & \frac{D}{D_{1}}\left[\Phi \left(\hat{H} (X^{k},X^{k-1},\bm{\delta}^{k})-\zeta\right) - \Phi \left(\hat{H} (X^{k+1},X^{k},\bm{\delta}^{k+1})-\zeta\right)\right] \\
			& + \frac{1}{4} \left(\|X^{k-2}-X^{k-1}\|_{F} - \|X^{k-1}-X^{k}\|_{F} + \|\bm{\delta}^{k-1}\|_{1}-\|\bm{\delta}^{k}\|_{1}\right).
		\end{aligned}
	\end{equation} 
	Summing up both sides of \eqref{eq:thirdbound} from $l = t$ to $k$, we have
	\begin{equation}\begin{aligned}
			&\ \frac{1}{2}\sum\limits_{l=t}^{k}\|X^l-X^{l-1}\|_{F}  \\
			\le&\ \frac{D}{D_{1}}\left[\Phi \left(\hat{H} (X^{t },X^{t -1},\bm{\delta}^{t })-\zeta\right) - \Phi \left(\hat{H} (X^{k+1},X^k,\bm{\delta}^{k+1})-\zeta\right)\right] \\
			&\ + \frac{1}{4} \left(\|X^{t -2}-X^{t -1}\|_{F} - \|X^{k -1}-X^k\|_{F} + \|\bm{\delta}^{t -1}\|_{1}-\|\bm{\delta}^{k}\|_{1}\right).
	\end{aligned}\end{equation}
	We then have $\bm{\delta}^{k}\to 0$ and $\|X^{k}-X^{k-1}\|_{F}\to 0$ by taking $k\to+\infty$ according to Lemma \ref{lem_Dec_acc} and \Cref{thm.update}(ii), respectively. Furthermore, we have from taking $k\to+\infty$ and continuity of $\Phi$ that $\Phi \left(\hat{H} (X^{k+1},X^k,\bm{\delta}^{k+1})-\zeta\right)\to\Phi (\zeta-\zeta)= \Phi(0)=0$ by \Cref{Def_desingularizing}(i). Therefore, for any $t > 0$, it holds that 
	\begin{equation}\label{eq_upbound_accsum}
		\begin{aligned} 
			T^{t } & =\sum\limits_{l=t }^{+\infty}\|X^l-X^{l-1}\|_{F} \\
			&\le  \frac{2D}{D_{1}}\Phi \left(\hat{H} (X^{t },X^{t -1},\bm{\delta}^{t })-\zeta\right) + \frac{1}{2} \left(\|X^{t -2}-X^{t -1}\|_{F} + \|\bm{\delta}^{t -1}\|_{1}\right) < +\infty,
		\end{aligned}
	\end{equation}   
	implying that $\{X^{k}\}$ is a Cauchy sequence and consequently converges uniquely. This completes the proof of all statements.
\end{proof}

Now we are ready to prove the convergence rate under the KL property. We mention that the ideas for the proof much follow those presented in \cite{attouch2009convergence,wen2018proximal}.

\begin{theorem}[Local convergence rate] 
	Let Assumptions \ref{ASM_fLip}-\ref{ASM_KL_acc} hold. Let $\{X^{k}\} $ be generated by \Cref{alg_acc} and converge to a critical point  $X^{*} \in \mathrm{crit} (F)$. Consider a desingularizing function of the form $\Phi(s) = cs^{1-\theta}$ where $c >0$ and {\L}ojasiewicz exponent $\theta \in [0,1)$. Then the following statements hold.
	\begin{itemize}
		\item[(i)] If $\theta=0 $, then there exist $\hat{k} \in \mathbb{N}$ such that $X^{k}\equiv X^{*} $ for all $k > \hat{k}$.
		\item[(ii)] If $\theta\in (0,\frac{1}{2}] $, there exist $\gamma\in (0,1) $ and $c_{0}, c_1>0 $ such that 
		\begin{equation}
			\|X^{k}-X^{*}\|_{F} \le c_{0}\gamma^{k} - c_{1}\|\bm{\delta}^{k}_{\mathcal{I}^{*}}\|_{1}
		\end{equation}
		for all sufficiently large $k $.
		\item[(iii)] If $\theta\in (\frac{1}{2},1) $, there exist $d_{0}, d_1>0 $ such that 
		\begin{equation}
			\|X^{k}-X^{*}\|_{F} \le d_{0}k^{-\frac{1-\theta}{2\theta-1}} - d_{1}\|\bm{\delta}^{k}_{\mathcal{I}^{*}}\|_{1}
		\end{equation}
		for all sufficiently large $k$. 
	\end{itemize}   
\end{theorem}

\begin{proof} 
	Since $X^{k}\to X^{*} $ as $k \to +\infty$, we have from Lemma \ref{lem_KLacc_cvgc_cond}(iv) and \eqref{eq_upbound_accsum} that
	\begin{equation}
		\|X^{k}-X^{*}\|_{F} = \|X^{k}-\lim\limits_{t\to+\infty}X^{t}\|_{F} = \|\lim\limits_{t\to+\infty}\sum\limits_{l=k}^{t} (X^{l}-X^{l+1})\|_{F}\le T^{k},
	\end{equation}
	and 
	\begin{equation}\label{eqKLacc_Rrbound}
		T^{k}\le \frac{2D}{D_{1}} \Phi \left(\hat{H} (X^{k},X^{k-1},\bm{\delta}^{k})-\zeta\right)
		+\frac{1}{2} (T^{k-2}-T^{k-1}) +\frac{1}{2}\|\bm{\delta}^{k-1}\|_{1}.
	\end{equation} 

	(i) 
	If $\theta=0 $, then $\Phi(s)=cs$ and  $\Phi{'} (s)=c$. we claim that there exists $\hat{k} \in \mathbb{N}$ such that $\hat{H} (X^{\hat{k}},X^{\hat{k}-1},\bm{\delta}^{\hat{k}}) =\zeta $. Seeking a contradiction, suppose this is not true, so $\hat{H} (X^{k}, X^{k-1},\bm{\delta}^{k}) > \zeta$ for all $k\in\mathbb{N}$. Since $\lim_{k\to+\infty} X^{k} = X^{*}$ and $\{\hat{H} (X^{k},X^{k-1},\bm{\delta}^{k})\}$  monotonically decreases to $\zeta $ by Lemma \ref{lem_KLacc_cvgc_cond}(ii). We have from the KL inequality that for all sufficiently large $k$
	\begin{equation}
		\| \nabla \hat{H} (X^{k},X^{k-1},\bm{\delta}^{k})\|_F \ge \frac{1}{c} > 0
	\end{equation}   
	with $\Phi{'} (s)=c $.
	This is a contradiction with $\|\nabla \hat{H} (X^{k},X^{k-1},\bm{\delta}^{k})\| \to 0 $ by Lemma \ref{lem_KLacc_cvgc_cond}(i).
	Thus, there exists $\hat{k}\in\mathbb{N} $ such that $\hat{H} (X^{k},X^{k-1},\bm{\delta}^{k})=\hat{H} (X^{\hat{k}},X^{\hat{k}-1},\bm{\delta}^{\hat{k}} ) \equiv\zeta $ for all $k\ge \hat{k} $. Hence, we conclude from Lemma \ref{lem_KLacc_cvgc_cond}(ii) that $X^{k}=X^{*}=X^{\hat{k}} $ for all $k>\hat{k} $, {i.e.}, the sequence converges in a finite number of iterations. This proves statement (i). 
	
	(ii)-(iii) If $\theta\in (0,1) $, then $\Phi{'} (s)=c(1-\theta)s^{-\theta}$. If there exists $\hat{k}\in\mathbb{N} $ such that $ \hat{H} (X^{\hat{k}},X^{\hat{k}-1},\bm{\delta}^{\hat{k}}) =\zeta $, then we know from Lemma \ref{lem_KLacc_cvgc_cond}(ii) that $X^{k+1} = X^{k}$ for all $k >\hat{k}$, indicating $X^k \equiv X^{\hat{k}} \in \chi^{\infty}$.  Therefore,  we need only to consider the case in which $\hat{H} (X^{k}, X^{k-1},\bm{\delta}^{k})>\zeta $ for all $k\in\mathbb{N}$.  
	
	Note that \Cref{ASM_KL_acc} implies
	\begin{equation} \label{eqKLacc_KLASM1}
		c (1-\theta) (\hat{H} (X^{k},X^{k-1},\bm{\delta}^{k})-\zeta)^{-\theta}  \|\nabla \hat{H} (X^{k},X^{k-1},\bm{\delta}^{k})\|_{F} \ge 1,
	\end{equation}
	for all $k>\hat{k} $ from Lemma \ref{lem_KLacc_cvgc_cond}(iv).
	On the other hand, we obtain from Lemma \ref{lem_KLacc_cvgc_cond}(i) that
	\begin{equation}
		\|\nabla \hat{H} (X^{k}, X^{k-1}, \bm{\delta}^{k})\|_F  \le D  \left(T^{k-2}-T^{k}+ \|\bm{\delta}^{k-1}\|_{1}-\|\bm{\delta}^{k}\|_{1} \right).
	\end{equation}   
	This, together with \eqref{eqKLacc_KLASM1}, yields
	\begin{equation}\label{ineq_KLacc_power_sk}
		(\hat{H} (X^{k},X^{k-1},\bm{\delta}^{k})-\zeta)^{\theta} \le c D (1-\theta)  \left( T^{k-2}-T^{k}+\|\bm{\delta}^{k-1}\|_{1}-\|\bm{\delta}^{k}\|_{1}\right).
	\end{equation}
	Taking a power of $\frac{1-\theta}{\theta} $ to both sides of \eqref{ineq_KLacc_power_sk}, we have for all $k \geq \hat{k}$ that  
	\begin{equation}
		\begin{aligned}
			\Phi (\hat{H} (X^{k},X^{k-1},\bm{\delta}^{k})-\zeta) =& c (\hat{H} (X^{k},X^{k-1},\bm{\delta}^{k})-\zeta)^{1-\theta} \\
			\le&\ c\left[c D(1-\theta)  \left(T^{k-2}-T^{k}+\|\bm{\delta}^{k-1}\|_{1}-\|\bm{\delta}^{k}\|_{1} \right)\right]^{\frac{1-\theta}{\theta}}\\
			\le&\ c\left[c D(1-\theta) \left(T^{k-2}-T^{k}+\|\bm{\delta}^{k-1}\|_{1}\right)\right]^{\frac{1-\theta}{\theta}}.  
	\end{aligned} \end{equation}
	This, together with  \eqref{eqKLacc_Rrbound}, yields 
	\begin{equation}\label{KLacc_obj_uppersk}
		\begin{aligned}
			T^{k}\le& \nu_{1} \left(T^{k-2}-T^{k}+\|\bm{\delta}^{k-1}\|_{1}\right)^{\frac{1-\theta}{\theta}} +
			\frac{1}{2} \left(T^{k-2}-T^{k}+\|\bm{\delta}^{k-1}\|_{1}\right)
	\end{aligned}\end{equation}
	with $\nu_{1} = \frac{2cD}{D_1} [cD (1-\theta)]^{\frac{1-\theta}{\theta}} >0$.
	
	It then follows from \eqref{KLacc_obj_uppersk}  that 
	\begin{equation}\label{KLacc_uppersk_par}
		\begin{aligned}
			&\ T^{k}+\frac{\sqrt{\mu}}{1-\mu}\|\bm{\delta}^{k}\|_{1} \\
			\le &\ \nu_{1}  \left(T^{k-2}-T^{k}+\|\bm{\delta}^{k-1}\|_{1}\right)^{\frac{1-\theta}{\theta}} + \frac{1}{2} \left(T^{k-2}-T^{k}+\|\bm{\delta}^{k-1}\|_{1}\right)+\frac{\sqrt{\mu}}{1-\mu}\|\bm{\delta}^{k}\|_{1}\\
			\overset{(a)}{\leq} &\ \nu_{1}  \left(T^{k-2}-T^{k}+\|\bm{\delta}^{k-1}\|_{1}\right)^{\frac{1-\theta}{\theta}} + \frac{1}{2} \left(T^{k-2}-T^{k}+\|\bm{\delta}^{k-1}\|_{1}\right)+\frac{\mu}{1-\mu}\|\bm{\delta}^{k-1}\|_{1}\\
			\overset{(b)}{\leq} &\ \nu_{1}  \left(T^{k-2}-T^{k}+\|\bm{\delta}^{k-1}\|_{1}\right)^{\frac{1-\theta}{\theta}} + \nu_{2} \left(T^{k-2}-T^{k}+\|\bm{\delta}^{k-1}\|_{1}\right),
	\end{aligned} \end{equation}
	where inequality $(a)$ holds simply due to $ \delta_{i}^{k}\le\sqrt{\mu} \delta_{i}^{k-1}, i \in [r^{*}]$ and inequality $(b)$ with $\nu_{2} = \frac{1}{2}+\frac{\mu}{1-\mu} > 0$ is true because $T^{k-2}-T^{k} \ge 0$ and $\mu \in (0,1)$.
	
	Now consider $\theta\in (0,\frac{1}{2}]$.  It follows from Lemma \ref{lem_Dec_acc}(ii) and \Cref{thm.update}(ii) that
	\begin{equation*}
		\lim\limits_{k\to+\infty}T^{k-2}-T^{k}+\|\bm{\delta}^{k-1}\|_{1} =0. 
	\end{equation*}
	Since $\frac{1-\theta}{\theta}\ge 1$, we thus know for sufficiently large $k$ that
	\begin{equation}
		\left(T^{k-2}-T^{k}+\|\bm{\delta}^{k-1}\|_{1}\right)^{\frac{1-\theta}{\theta}} \le  \left( T^{k-2}-T^{k}+\|\bm{\delta}^{k-1}\|_{1}\right).
	\end{equation}  
	This, together with \eqref{KLacc_uppersk_par}, yields
	\begin{equation}\label{eq:oneBound}
		T^{k}+\frac{\sqrt{\mu}}{1-\mu}\|\bm{\delta}^{k}\|_{1} \le  (\nu_{1}+\nu_{2}) \left(T^{k-2}-T^{k}+\|\bm{\delta}^{k-1}\|_{1}\right).
	\end{equation}
	
	On the other hand, by \Cref{thm.update}(ii), we have that $ \delta_{i}^{k}\le\sqrt{\mu} \delta_{i}^{k-1}$ and $\delta_{i}^{k-1}\le\sqrt{\mu} \delta_{i}^{k-2}$, $ i \in [r^{*}]$, and it implies that  
	\begin{equation} \label{eps_delta_item_ineq}
		\delta_{i}^{k-1}\le \frac{\sqrt{\mu}}{1-\mu} ( \delta_{i}^{k-2}- \delta_{i}^{k}),\ i \in [r^{*}],
	\end{equation} 
	leading to
	\begin{equation}\label{ineq_norm_delta}
		\|\bm{\delta}^{k-1}\|_{1}\le \frac{\sqrt{\mu}}{1-\mu} (\|\bm{\delta}^{k-2}\|_{1}-\|\bm{\delta}^{k}\|_{1}).
	\end{equation}
	Therefore, we have from \eqref{eq:oneBound} for any $k\geq\hat{k}$ that 
	\begin{equation}
		\begin{aligned}
			&\ T^{k}+\frac{\sqrt{\mu}}{1-\mu}\|\bm{\delta}^{k}\|_{1} \\   
			\overset{(a)}{\leq} &\ (\nu_{1}+\nu_{2})  \left( T^{k-2} + \frac{\sqrt{\mu}}{1-\mu}\Vert \bm{\delta}^{k-2}\Vert_1 - \left( T^{k} + \frac{\sqrt{\mu}}{1-\mu}\Vert \bm{\delta}^{k}\Vert_1\right) \right) \\
			\leq &\ \left(\frac{\nu_{1}+\nu_{2}}{1+\nu_{1}+\nu_{2}}\right)\left( T^{k-2} + \frac{\sqrt{\mu}}{1-\mu}\Vert \bm{\delta}^{k-2}\Vert_1\right)  \\
			\le &\  \left(\frac{\nu_{1}+\nu_{2}}{1+\nu_{1}+\nu_{2}}\right)^{\lfloor\frac{k-\hat{k}}{2}\rfloor}
			\left(T^{[ (k-\hat{k})\,\mathrm{mod}\, 2]+\hat{k}}+\frac{\sqrt{\mu}}{1-\mu}\|\bm{\delta}^{[ (k-\hat{k})\,\mathrm{mod}\, 2]+\hat{k}}\|_{1}\right) \\
			\le &\ \left(\frac{\nu_{1}+\nu_{2}}{1+\nu_{1}+\nu_{2}}\right)^{\frac{k-\hat{k}}{2}} \left(T^{\hat{k}}+\frac{\sqrt{\mu}}{1-\mu}\|\bm{\delta}^{\hat{k}}\|_{1} \right),
	\end{aligned}\end{equation} 
	where inequality $(a)$ holds by \eqref{ineq_norm_delta}. 
	We hence have for any $k \geq \hat{k}$ that
	\begin{equation}
		\|X^{k}-X^{*}\|_{F} \le T^{k} \le c_{0}\gamma^{k}-c_{1}\|\bm{\delta}^{k}\|_{1}
	\end{equation} 
	with 
	\begin{equation*}
		\gamma=\sqrt{\frac{\nu_{1}+\nu_{2}}{\nu_{1}+\nu_{2}+1}},c_{0} = \frac{T^{\hat{k}}+\frac{\sqrt{\mu}}{1-\mu}\|\bm{\delta}^{\hat{k}}\|_{1}}{\gamma^{\hat{k}}},c_{1} = \frac{\sqrt{\mu}}{1-\mu}.
	\end{equation*}
	This completes the proof of statement (ii).
	
	Consider now $\theta\in (\frac{1}{2},1)$. We have from  $\frac{1-\theta}{\theta}<1$ and $\lim\limits_{k\to+\infty}T^{k-2}-T^{k}+\|\bm{\delta}^{k-1}\|_{1} =0$  for sufficiently large $k$  that
	\begin{equation}
		\left( T^{k-2}-T^{k}+\|\bm{\delta}^{k-1}\|_{1}\right) \le  \left(T^{k-2}-T^{k}+\|\bm{\delta}^{k-1}\|_{1}\right)^{\frac{1-\theta}{\theta}}.
	\end{equation}
	This, together with \eqref{KLacc_uppersk_par}, gives
	\begin{equation}\label{eq_Klacc_sumbound2}
		T^{k}+\frac{\sqrt{\mu}}{1-\mu}\|\bm{\delta}^{k}\|_{1} \le  (\nu_{1}+\nu_{2}) \left(T^{k-2}-T^{k}+\|\bm{\delta}^{k-1}\|_{1}\right)^{\frac{1-\theta}{\theta}}.
	\end{equation}
	Raising a power of $\frac{\theta}{1-\theta} $ to the both sides of \eqref{eq_Klacc_sumbound2} and considering \eqref{ineq_norm_delta} gives 
	\begin{equation}
		\left(T^{k}+\frac{\sqrt{\mu}}{1-\mu}\|\bm{\delta}^{k}\|_{1}\right)^{\frac{\theta}{1-\theta}} \le \nu_{3}\left[ (T^{k-2}+\frac{\sqrt{\mu}}{1-\mu}\|\bm{\delta}^{k-2}\|_{1})- (T^{k} +\|\bm{\delta}^{k}\|_{1})\right],
	\end{equation}
	where $\nu_{3}= \left(\nu_{1}+v_{2}\right)^\frac{\theta}{1-\theta} $.
	
	Split the sequence $\{k_{2},k_{2}+1,\cdots\}$ into even and odd subsequences. For the even subsequnce, define $ \Delta_{t}:=T^{2t}+\frac{\sqrt{\mu}}{1-\mu}\|\bm{\delta}^{2t}\|_{1} $ for $t\ge \lceil\frac{\hat{k}}{2}\rceil:=N_{1} $. Following from the techniques presented in the proofs of \cite[Theorem 4]{Zeng_Acc_2022} and \cite[Theorem 2]{attouch2009convergence},  we have
	\begin{equation}\label{ineq_KLacc_cvgcrate2}
		\Delta_{k} \le  \left( \Delta_{N_{1}-1}^{\frac{1-2\theta}{1-\theta}} + \nu (k-N_{1})\right)^{-\frac{1-\theta}{2\theta-1}} \le d_{2}k^{-\frac{1-\theta}{2\theta-1}} ,
	\end{equation}
	for some $d_{2}>0 $. As for the odd subsequence of $\{k_{2},k_{2}+1,\cdots\} $,  define $ \Delta_{t}:=T^{2t+1}+\frac{\sqrt{\mu}}{1-\mu}\|\bm{\delta}^{2t+1}\|_{1} $. We know that \eqref{ineq_KLacc_cvgcrate2} still holds.  
	Therefore, for all sufficiently large and even number $k $, it holds that
	\begin{equation}
		\|X^{k}-X^{*}\|_{F}\le T^{k}=  \Delta_{\frac{k}{2}}-\frac{\sqrt{\mu}}{1-\mu}\|\bm{\delta}^{k}\|_{1} \le 2^{\frac{1-\theta}{2\theta-1}}d_{2} k^{-\frac{1-\theta}{2\theta-1}}-\frac{\sqrt{\mu}}{1-\mu}\|\bm{\delta}^{k}\|_{1}.
	\end{equation}
	
	For all sufficiently large and odd numbers $k $,
	\begin{equation}
		\|X^{k}-X^{*}\|_{F}\le T^{k} =  \Delta_{\frac{k-1}{2}}-\frac{\sqrt{\mu}}{1-\mu}\|\bm{\delta}^{k}\|_{1} \le 2^{\frac{1-\theta}{2\theta-1}}d_{2}  (k-1)^{-\frac{1-\theta}{2\theta-1}}-\frac{\sqrt{\mu}}{1-\mu}\|\bm{\delta}^{k}\|_{1}.
	\end{equation}
	
	Overall, we thus have for any sufficiently large $k $  that
	\begin{equation}
		\|X^{k}-X^{*}\|_{F}\le d_{0} k^{-\frac{1-\theta}{2\theta-1}}-d_{1}\|\bm{\delta}^{k}\|_{1},
	\end{equation}
	where
	$
	d_{0} =2^{\frac{1-\theta}{2\theta-1}}d_{2}\max \left(1,2^{\frac{1-\theta}{2\theta-1}}\right) \,\text{and}\, d_{1} =\frac{\sqrt{\mu}}{1-\mu}.
	$ 
	The proof is complete.
\end{proof}


\section{Numerical Experiments}\label{Sec_numerical}
In this section, we conduct low-rank matrix completion tasks on both synthetic data and natural images to showcase the effectiveness and efficiency of the proposed EIRNRI algorithm. The presented examples revolve around addressing the matrix completion problem, defined as follows:
\begin{equation}\label{exp_pro_MCrand}
	\min\limits_{X\in\mathbb{R}^{m\times n}}\, \frac{1}{2}\|M-\mathbf{P}_{\Omega} (X)\|_{F}^{2} + \lambda\|X\|_{p}^{p},
\end{equation}
where $M\in\mathbb{R}^{m\times n} $, $\Omega $ is the set of indices of samples, and $\mathbf{P}_{\Omega}$ denotes the projection onto the subspace of sparse matrices with nonzeros restricted to the index set $\Omega$. In this problem, the Lipschitz constant $L_{f} = 1$. 

The codes of all methods tested in this section are implemented in Matlab on a desktop with an Intel Core i5-8500 CPU  (3.00 GHz) and 24GB RAM running 64-bit Windows 10 Enterprise. We have made the code publicly available at \url{https://github.com/Optimizater/Low-rank-optimization-with-rank-identification}.

In our experiments, the relative error is defined as
\begin{equation}
	\text{RelErr} = \frac{\|X^{k}-X^{*}\|_{F} }{\|X^{*}\|_{F}}.
\end{equation}
To verify whether $X$ is a critical point of problem \eqref{exp_pro_MCrand},  we use the distance defined as $\text{dist} (0,\partial F (X)) = \|U_{r}^{\top}\nabla f (X)V_{r}+\lambda p \Sigma^{p-1} \|_{F}$ define the relative distance error as $\text{RelDist} \le \frac{\text{dist} (0,\partial F (X))}{\|M\|_{F}}.$
We adopt the same criterion as used in  \cite{ExpCriterion_2010,opt_simu_svd_2017}, and say a matrix $X^{*} $ is successfully recovered by $X^{k}$ if the corresponding relative error or the relative distance is less than $10^{-5}$.

\subsection{Synthetic Data}
For synthetic data, our target is to recover the generated matrix $X^{*}$ from $M$. Thus, we call a correct low-rank detecting (CLD) that the rank of the termination point is the same as the rank of $X^{*}$. 
In our experiments, we compare the performance of IRNRI and EIRNRI with PIRNN \cite{opt_simu_svd_2017}  for solving \eqref{exp_pro_MCrand} on random data, where IRNRI does not consider the extrapolation technique.
We set $m = n = 150,r=5,10,15$ for the matrix $X^{*}$, select $p=0.5$ for Schatten-$p$ norm.  The rank $r$ matrix $X^{*}$ is generated by $X^{*} = BC$, where $B\in\mathbb{R}^{m\times r}$ and $C\in\mathbb{R}^{r\times n}$ are generated randomly with {i.i.d.} standard Gaussian entries. 
We randomly sample a subset $\Omega$ with sampling ratio $\text{SR}$($|\Omega|$ satisfies $|\Omega| = \lceil\text{SR}*{(mn)}\rceil$), and select three values $0.2, 0.5$ and $0.8$, and then the observed matrix is $M = \mathbf{P}_{\Omega} (X^{*})$.

All algorithms are initialized with a random Gaussian matrix $X^{0}$, and terminate when $\text{RelErr}\le\textit{opttol}$ or $\text{RelDist}\le\textit{opttol}$ or the number of iteration exceeds the preset limit $\textit{itmax}$. Furthermore, we add another termination condition $\|X^{k+1}-X^{k}\|_{\infty} \le \textit{KLopt}$ by Lemma \ref{lem_KLacc_cvgc_cond}.
Unless otherwise mentioned, we use the following parameters to run the experiments: $\beta=1.1 > L_f$,  $\epsilon_{i}^{0}=1,i\in[m]$, $\mu=0.1$ for IRNRI and EIRNRI, $\textit{opttol}=10^{-5} $, $\textit{itmax}=1\times 10^{3}$ and $\textit{KLopt} = 10^{-7}$.  For EIRNRI,  $\alpha = 0.7$ is roughly tuned for the best performance correspondingly. 
For PIRNN,  $\epsilon_i, i\in[m]$ are fixed as sufficiently small values: $\epsilon_i =  10^{-3}$.  

The number of problems (out of 2000 problems in total) converging to a solution with the correct rank satisfying ${\Rank}(X_{true}) = {\Rank}(X^{*}):=r^{*}$ are depicted 
in Figure \ref{fig_MatrixIden}.  The average of the final relative errors of each algorithm for each case 
are shown in Table \ref{Tab_MatrixIden}. 
We can observe from these results that our proposed algorithms can always return 
a solution with low relative error and the correct rank.

\begin{figure}[H]
	\centering{
		\subfigure[The number of CLD with $r^{*}=5$.]{
			\includegraphics[width=0.3\linewidth]{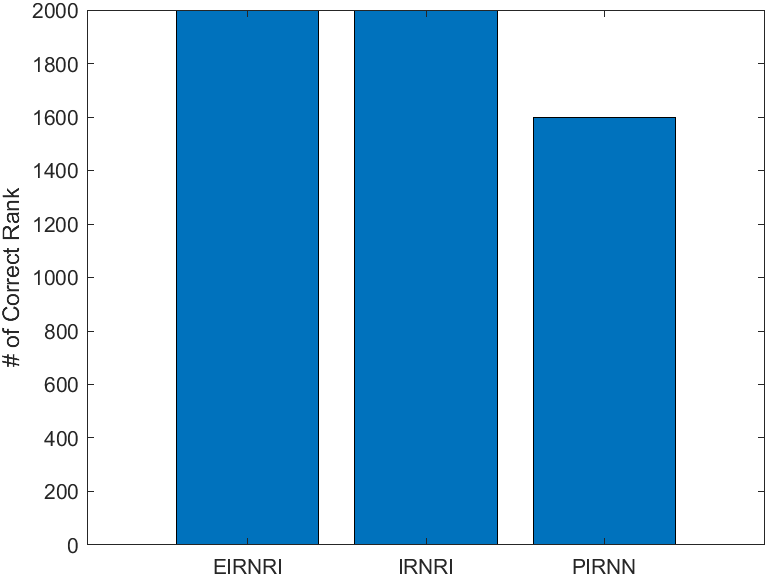}
		}
		\subfigure[The number of CLD with $r^{*}=10$.]{
			\includegraphics[width=0.3\linewidth]{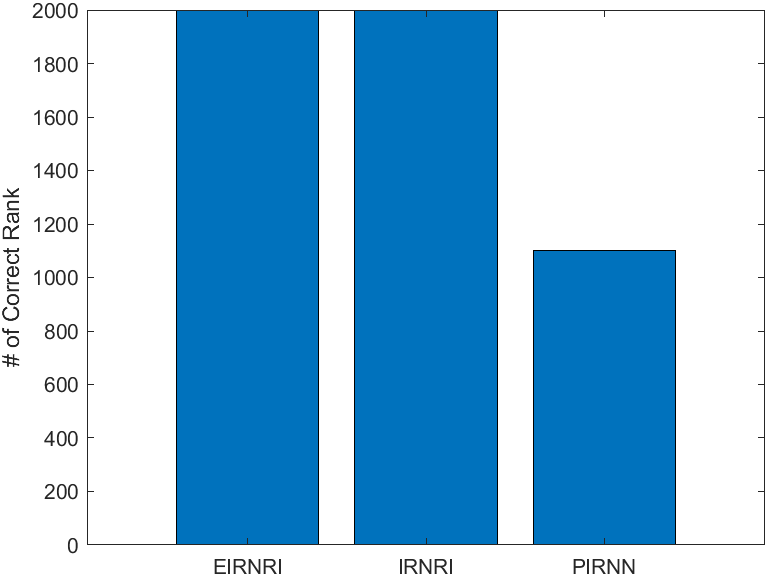}
		}
		\subfigure[The number of CLD with $r^{*}=15$.]{
			\includegraphics[width=0.3\linewidth]{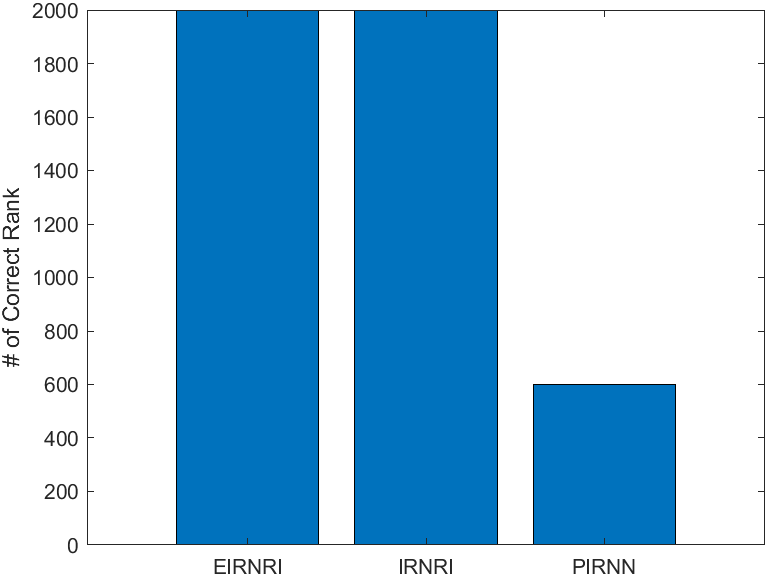}
		}
	}
	\caption{The number of problems that achieve model identification for different $r^{*}$. For each $r^{*}$, 2000 problems are generated and set $\lambda = 10^{-1}\|X^{*}\|_{\infty}$ for those problems. We set different $\epsilon = 10^{-3}$ for PIRNN and the default values for our algorithm.}
	\label{fig_MatrixIden}
\end{figure}

\begin{table}[htbp]
	\centering
	\begin{tabular}{cccc}
		\toprule
		& $r^* = 5$ & $r^*=10$ & $r^* = 15$ \\
		\hline
		IRNRI  & $6.42\times 10{-7}$   & $2.67\times 10^{-5}$ & $9.15\times 10^{-4}$   \\
		EIRNRI & $6.42\times 10{-7}$   & $2.67\times 10^{-5}$ & $9.15\times 10^{-4}$   \\
		PIRNN  & $5.13\times 10^{-6}$  & $3\times 10^{-4}$    & $1.1\times 10^{-2}$   \\
		\bottomrule
	\end{tabular}
	\caption{The average of the final relative errors of each algorithm with different $r^{*}$.}
	\label{Tab_MatrixIden}
\end{table}

To compare the efficiency of the algorithms,  we also plot the evolution of the relative distance for each case in Figure \ref{fig_LLrate}.
We can see that the proposed IRNRI achieves roughly the same speed as PIRNN, while the proposed  EIRNRI significantly outperforms PIRNN.  


\begin{figure}[H]
	\centering{
		\subfigure[$r^*=5$]{
			\includegraphics[width=0.3\linewidth]{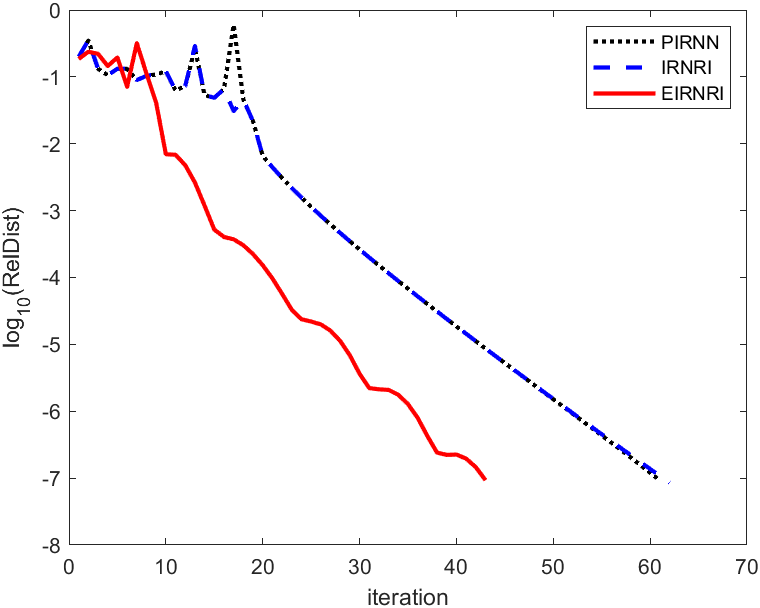}
		}
		\subfigure[$r^*=10$]{
			\includegraphics[width=0.3\linewidth]{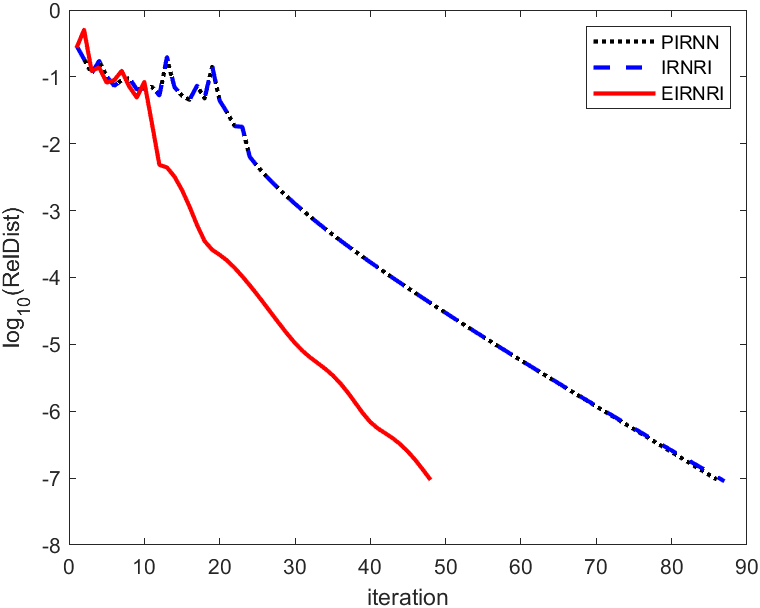}
		}
		\subfigure[$r^*=15$]{
			\includegraphics[width=0.3\linewidth]{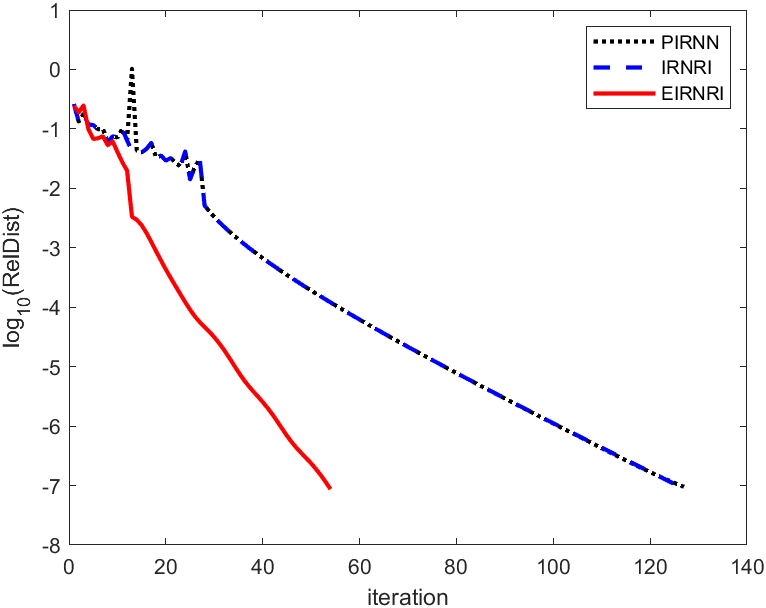}
		}
	}
	\caption{Comparison of matrix recovery on synthetic data with initial point $X^{0}$ satisfies ${\Rank}(X^{0}) = r^*$, initial parameters satisfy $\epsilon_{i}^{0} = \epsilon = 10^{-3}$. 
	}
	\label{fig_LLrate}
\end{figure}

Since Algorithm \ref{alg_acc} involves a new parameter $\alpha$ to accelerate the convergence, 
we set the different values of $\alpha =\{ 0,0.1,0.3,0.5,0.7,0.9\}$ to test the sensitivity to $\alpha_{k}$ for the EIRNRI algorithm. 
The experiments show the best value is always around $\alpha= 0.7$.
\begin{figure}[H]
	\centering{
		\subfigure[$r^*=5$]{
			\includegraphics[width=0.3\linewidth]{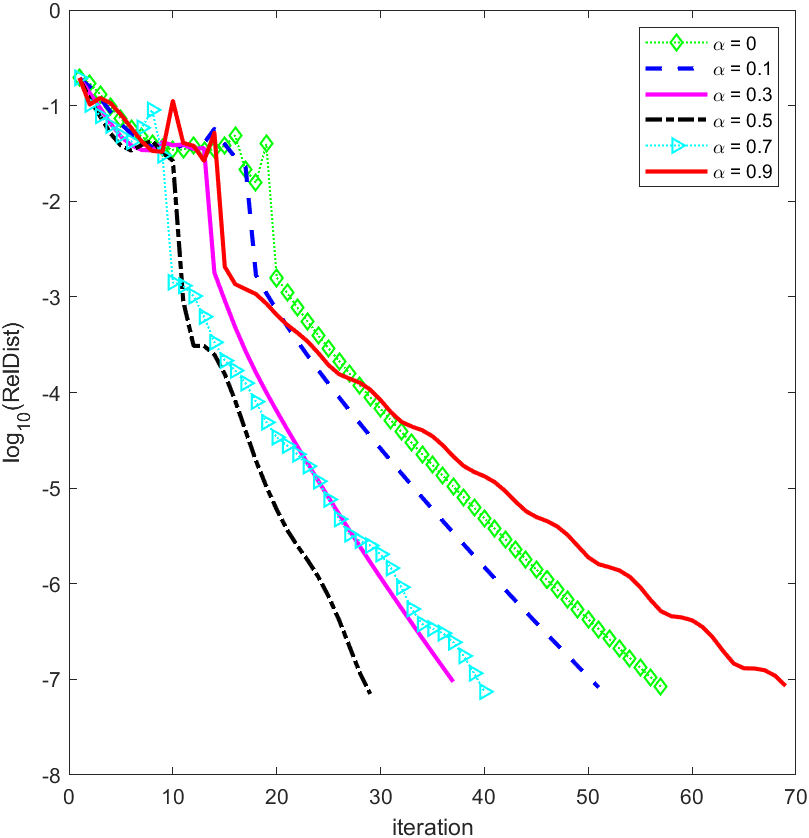}
		}
		\subfigure[$r^*=10$]{
			\includegraphics[width=0.3\linewidth]{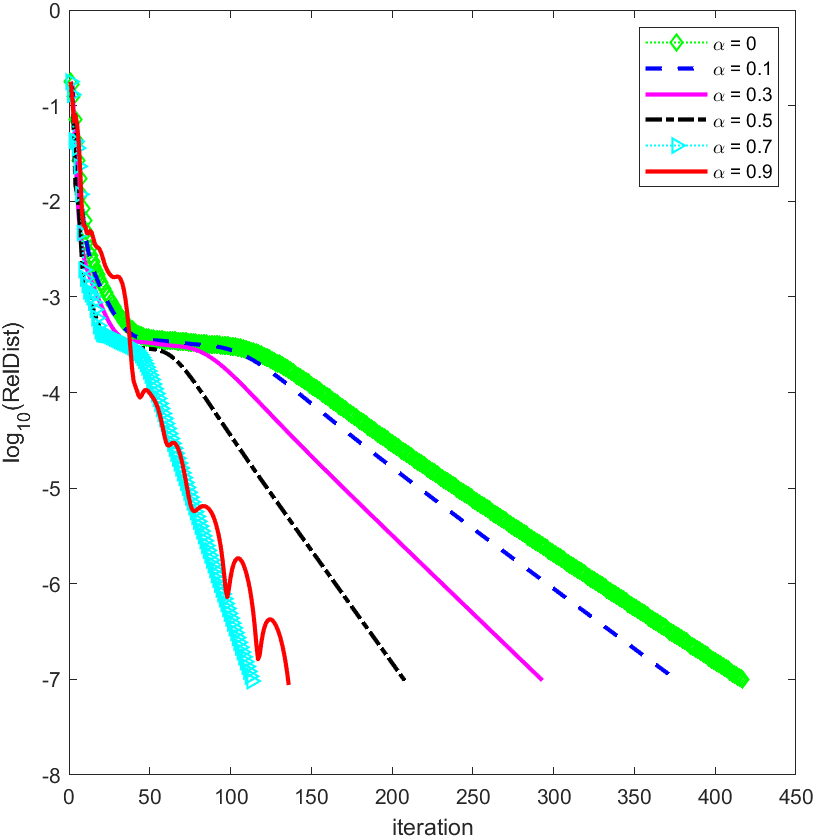}
		}
		\subfigure[$r^*=15$]{
			\includegraphics[width=0.3\linewidth]{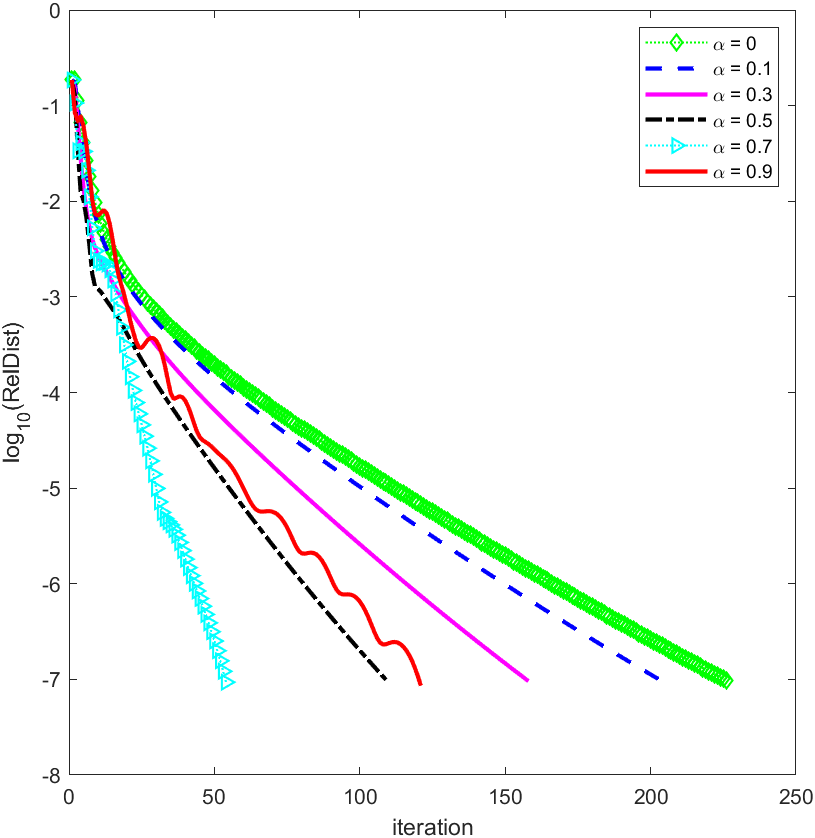}
		}
	}
	\caption{The performance of EIRNRI with different $\alpha$ for three kinds of problems.
	}
	\label{fig_stvAlp}
\end{figure}

\subsection{Application to Image Recovery}

In this section, we compare  our method with four solvers, PIRNN \cite{opt_simu_svd_2017}, AIRNN \cite{Alg_conflict_AIRNN_2021}\footnote{\href{https://github.com/ngocntkt/AIRNN}{Available at https://github.com/ngocntkt/AIRNN}},
Sc$p$ \cite{Relax_NCVX_CNN_RPCA_2013}\footnote{\href{https://github.com/liguorui77/scpnorm}{Available at https://github.com/liguorui77/scpnorm}
},  
and FGSR$p$ \cite{LRMR_GroupSparse_2019}\footnote{\href{https://github.com/udellgroup/Codes-of-FGSR-for-effecient-low-rank-matrix-recovery}{Available at https://github.com/udellgroup/Codes-of-FGSR-for-effecient-low-rank-matrix-recovery}
}.
We consider the real images though they are usually not low rank, but the top singular values of the real images dominate the main information. The nature images are of scale $300 \times 300 \times 3 $, we observe the top singular values first and sample the $80\% $ of the elements uniformly for the random mask.  The block mask is also been used.   We set $\lambda = 0.5$ as the regularization parameter. 
The iteration is terminated when  $\text{RelDist}\le{10}^{-5}$ or $\textit{KLopt}\le 10^{-5} $ or the number of iteration exceeds \textit{itmax}$=10^{3}$.  
We set  $\beta = 1.1$ for the three algorithms.
In order to guarantee the accuracy of the solution returned by  PIRNN,  we set $\epsilon = 10^{-4} $ for PIRNN, and  $\bm{\epsilon}^{0} = 1, \mu=0.1 $ for IRNRI and EIRNRI. As for EIRNRI, we set $\alpha=0.7$. The default parameter values for other algorithms are used.
The performance of all algorithms is compared on two measurements:  
(1) the difference of the rank of the final iterate  and $X^{*}$,
(2) the peak signal-to-noise ratio (PSNR), defined as 
\begin{equation}
	\text{PSNR} (M,X^{*}) = 10 \times \log_{10}\frac{255^{2}}{\frac{1}{3mn}\sum_{i=1}^{3}\|X_{i}^{*}-M_{i}\|_{F}^{2}}, 
\end{equation}
where  $M $ is the restored image and $X^{*} $  is the original image.

\begin{figure}[htbp]
	\centering{
		\subfigure[Original image]{\includegraphics[width=0.2\linewidth]{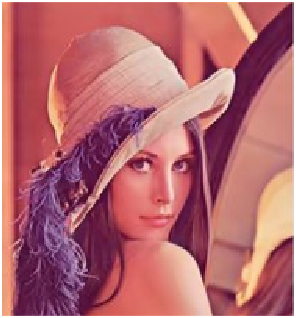}}
		\subfigure[Low-rank image]{\includegraphics[width=0.2\linewidth]{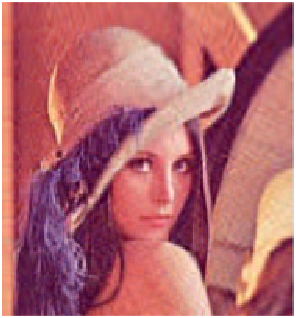}}
		\subfigure[Random mask]{\includegraphics[width=0.2\linewidth]{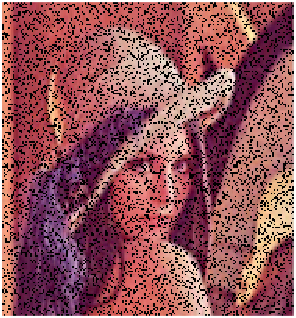}}
		\subfigure[PIRNN]{\includegraphics[width=0.2\linewidth]{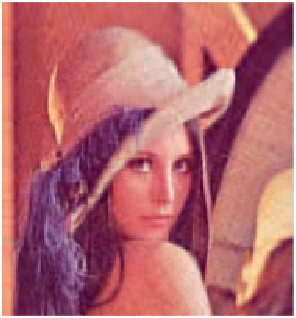}}
		\subfigure[AIRNN]{\includegraphics[width=0.2\linewidth]{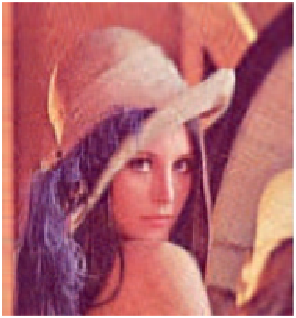}}
		\subfigure[EIRNRI]{\includegraphics[width=0.2\linewidth]{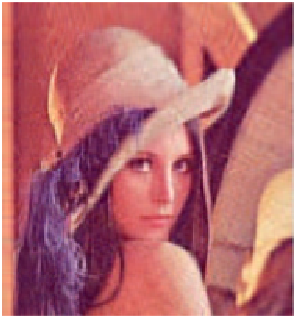}}
		\subfigure[Sc$p$]{\includegraphics[width=0.2\linewidth]{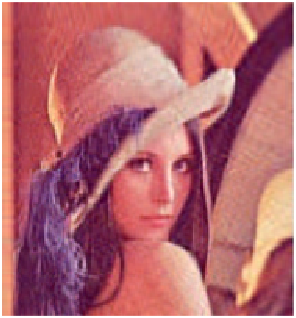}}
		\subfigure[FGSR$p$]{\includegraphics[width=0.2\linewidth]{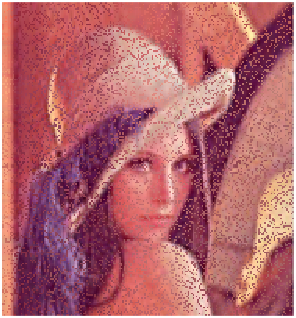}}
	}
	\caption{The performance of different methods in image recovery with a random mask(SR= $0.8$).
		(a) Original image, ${\Rank}(X) = 300$; 
		(b) Low-rank image, ${\Rank}(X^{*}) = 30$;
		(c) Noised picture;
		(d) PIRNN:   ${\Rank}(X) = 30$,  PSNR=$28.855$;
		(e) AIRNN:   ${\Rank}(X) = 36$,  PSNR=$28.854$;   
		(f) EIRNRI:  ${\Rank}(X) = 30$,  PSNR=$28.855$;    
		(g) Sc$p$:   ${\Rank}(X) = 187$,  PSNR=$28.971$;   
		(h) FGSR$p$: ${\Rank}(X) = 187$,  PSNR=$21.806$;   
	}
	\label{Relimg_Lena}
\end{figure}

\begin{table}[bthp]
	\begin{tabular}{ccccccccccc}
		\toprule
		& \multicolumn{2}{c}{PIRNN} & \multicolumn{2}{c}{AIRNN} & \multicolumn{2}{c}{EIRNRI} & \multicolumn{2}{c}{Sc$p$} & \multicolumn{2}{c}{FGSR$p$} \\
		\hline
		& PSNR      & Rank          & PSNR          & Rank       & PSNR      & Rank           & PSNR             & Rank & PSNR         & Rank      \\
		\hline
		$r^{*}=15$ & 24.971    & \textbf{15}   & 24.107        & 29         & 24.971    & \textbf{15}    & \textbf{24.981}  & 187  & 20.918       & 187       \\
		$r^{*}=20$ & 26.547    & \textbf{20}   & 26.37         & 24         & 26.547    & \textbf{20}    & \textbf{26.571}  & 187  & 21.433       & 187       \\
		$r^{*}=25$ & 27.805    & \textbf{25}   & 27.804        & 25         & 27.805    & \textbf{25}    & \textbf{27.86}   & 187  & 21.687       & 187       \\
		$r^{*}=30$ & 28.86     & \textbf{30}   & 28.263        & 36         & 28.86     & \textbf{30}    & \textbf{28.975}  & 187  & 21.885       & 187       \\
		$r^{*}=35$ & 29.845    & \textbf{35}   & 27.614        & 47         & 29.846    & \textbf{35}    & \textbf{30.063}  & 187  & 22.064       & 187       \\
		$r^{*}=40$ & 30.715    & \textbf{40}   & 30.019        & 43         & 30.716    & \textbf{40}    & \textbf{31.099}  & 187  & 22.133       & 187       \\
		\bottomrule
	\end{tabular}
	\caption{The performance of different methods to recover random mask images. Bold values correspond to the best results for each task.}
	\label{tab_img_recovery_rand_Lena}
\end{table}

Figure \ref{Relimg_Lena} depicts the recovered results for random cover,  and Table \ref{tab_img_recovery_rand_Lena} shows 
the quality of each recovered result.  We can see that all those algorithms can recover the image under the random mask, and our proposed algorithm can always find a lower rank than the others. 

Figure \ref{Relimg_R1} depicts the recovered results for block mask, and  Table \ref{tab_img_recovery_block_R1} shows 
the quality of each recovered result. We can observe that  AIRNN can pursue low-rank targets but it can not achieve a credible recovery for block mask image recovery.  Though Sc$p$ makes the biggest PSNR in both two tasks, the restored image always has a higher rank than the low-rank target image. EIRNRI restores a low-rank image with a similar PSNR with Sc$p$, indicating that it can recover an accurate low-rank structure of the image.

\begin{figure}[ht]
	\centering{
		\subfigure[Original image]{\includegraphics[width=0.2\linewidth]{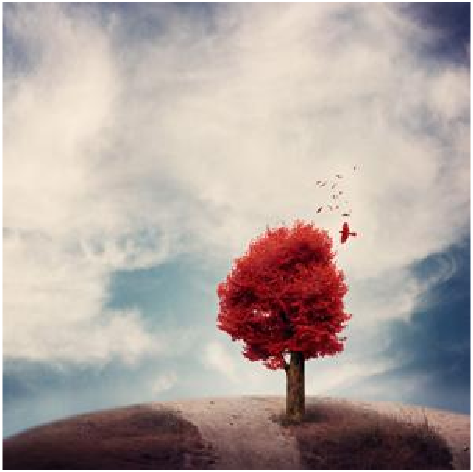}}
		\subfigure[Low-rank image]{\includegraphics[width=0.2\linewidth]{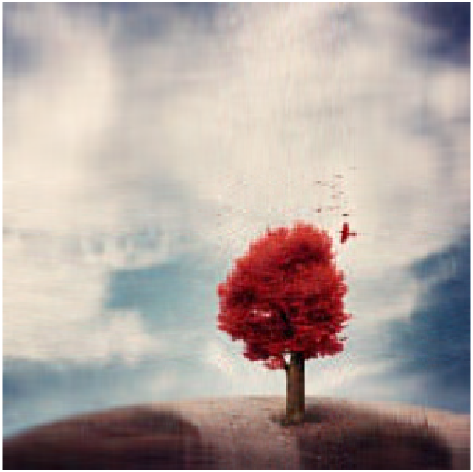}}
		\subfigure[Random mask]{\includegraphics[width=0.2\linewidth]{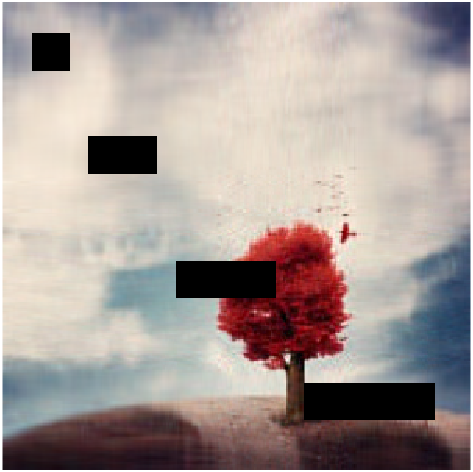}}
		\subfigure[PIRNN]{\includegraphics[width=0.2\linewidth]{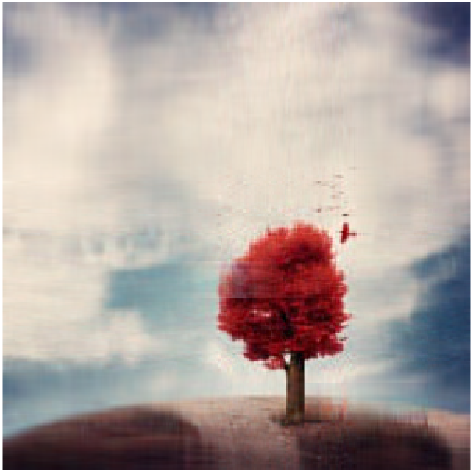}}
		\subfigure[AIRNN]{\includegraphics[width=0.2\linewidth]{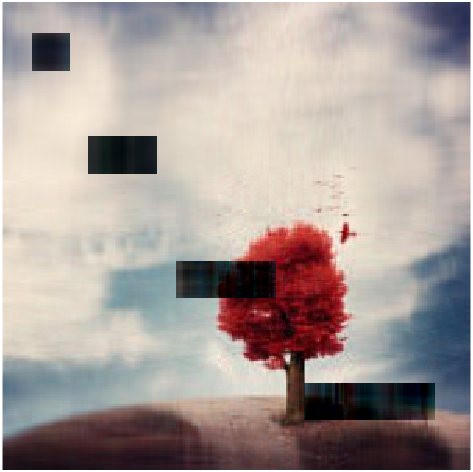}}
		\subfigure[EIRNRI]{\includegraphics[width=0.2\linewidth]{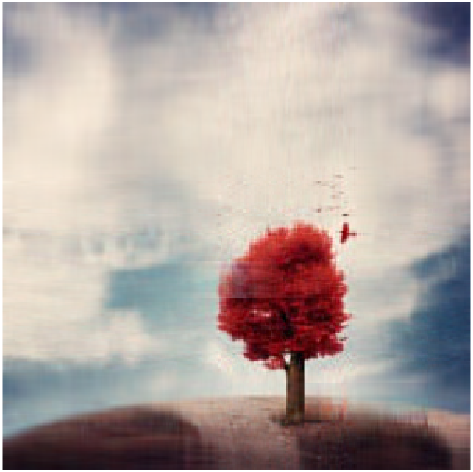}}
		\subfigure[Sc$p$]{\includegraphics[width=0.2\linewidth]{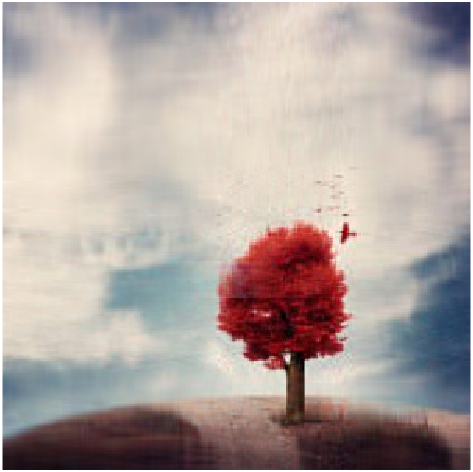}}
		\subfigure[FGSR$p$]{\includegraphics[width=0.2\linewidth]{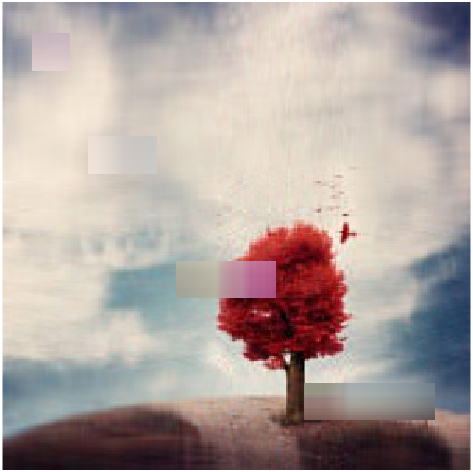}}
	}
	\caption{The performance of different methods in image recovery with block mask. 
		(a) Original image, ${\Rank}(X) = 300$; 
		(b) Low-rank image, ${\Rank}(X^{*}) = 30$;
		(c) Noised picture; 
		(d) PIRNN:   ${\Rank}(X) = 30$,  PSNR=$32.696$;
		(e) AIRNN:   ${\Rank}(X) = 32$,  PSNR=$16.423$;
		(f) EIRNRI:  ${\Rank}(X) = 30$,  PSNR=$32.654$;
		(g) Sc$p$:   ${\Rank}(X) = 126$,  PSNR=$33.237$;
		(h) FGSR$p$: ${\Rank}(X) = 126$,  PSNR=$23.132$;
	}
	\label{Relimg_R1}
\end{figure}

\begin{table}[htbp]
	\begin{tabular}{ccccccccccc}
		\toprule
		& \multicolumn{2}{c}{PIRNN} & \multicolumn{2}{c}{AIRNN} & \multicolumn{2}{c}{EIRNRI} & \multicolumn{2}{c}{Sc$p$} & \multicolumn{2}{c}{FGSR$p$} \\
		\hline
		& PSNR         & Rank       & PSNR          & Rank       & PSNR      & Rank           & PSNR            & Rank  & PSNR         & Rank      \\
		\hline
		$r^{*}=15$ & 30.848    & \textbf{15}   & 16.276        & 19         & 30.844    & \textbf{15}    & \textbf{31.003}  & 75   & 26.307       & 75        \\
		$r^{*}=20$ & 31.893    & \textbf{20}   & 16.305        & 24         & 31.887    & \textbf{20}    & \textbf{32.185}  & 100  & 26.687       & 100       \\
		$r^{*}=25$ & 32.468    & \textbf{25}   & 16.356        & 28         & 32.448    & \textbf{25}    & \textbf{32.881}  & 121  & 26.946       & 121       \\
		$r^{*}=30$ & 32.696    & \textbf{30}   & 16.423        & 32         & 32.654    & \textbf{30}    & \textbf{33.237}  & 126  & 27.132       & 126       \\
		$r^{*}=35$ & 32.794    & \textbf{35}   & 16.433        & 36         & 32.703    & \textbf{35}    & \textbf{33.378}  & 131  & 27.272       & 131       \\
		$r^{*}=40$ & 32.465    & \textbf{38}   & 16.502        & 39         & 32.357    & \textbf{38}    & \textbf{33.044}  & 136  & 27.38        & 136       \\
		\bottomrule
	\end{tabular}
	\caption{The performance of different methods in image recovery with block mask. Bold values correspond to the best results for each task.
	}
	\label{tab_img_recovery_block_R1}
\end{table}

\section{Conclusion}\label{Sec_Conclusion}
In this paper, we have proposed, analyzed, and implemented iteratively reweighted Nuclear norm methods 
for solving the Schatten-$p$ norm regularized low-rank optimization. 
Our work features two main novelties. The first is the exhibition of a rank identification property, enabling the algorithm to identify the rank of stationary points of the concerned problems in finite iterations. Leveraging this property, we  have designed a novel updating strategy 
for $\epsilon_i$, so that $\epsilon_i$ associated with the positive 
singular values can be driven to zero rapidly and those associated with zero singular values can be 
automatically fixed as constants after a finite number of iterations. The crucial role of this strategy is that the algorithm 
eventually behaves like a truncated weighted Nuclear norm method so that the techniques for smooth 
algorithms can be directly applied including acceleration techniques and convergence analysis. 

The convergence properties established for our algorithm are illustrated empirically on 
test sets comprising both synthetic and real datasets. We remark, however, that several practical considerations remain to be addressed for enhancing the performance of the proposed method.  One potential avenue for improvement involves integrating the rank identification property into the implementation. Once the correct rank has been identified within a finite number of iterations, the algorithm can be terminated and subsequently switched to a traditional Frobenius recovery with a fixed rank to further enhance the quality of the recovered solution. We defer the exploration of this aspect to future research endeavors. 


\bibliographystyle{plain}
\bibliography{references}

\end{document}